\documentclass[a4paper]{amsart}%[reqno]{amsart}
\usepackage{amssymb,amsmath,amsthm,amstext,amsfonts}
\usepackage[dvips]{graphicx}
\usepackage{psfrag}
\usepackage{url}
\usepackage{amsfonts}
\usepackage{amsmath}
\usepackage{mathrsfs}
\usepackage{graphicx}
\usepackage{amsmath,amsthm,amssymb,amscd}
\usepackage{color}
\usepackage{enumerate}
\usepackage[colorlinks, linkcolor=blue, anchorcolor=blue, citecolor=blue]{hyperref}

\usepackage{epsfig}
\usepackage{tikz}  
\usetikzlibrary{quotes,arrows.meta,graphs,positioning}

\makeatletter \@addtoreset{equation}{section} \makeatother

\renewcommand\thetable{\thesection.\@arabic\c@table}

\theoremstyle{plain}

\newtheorem{theorem}{Theorem}[section]

\newtheorem{lemma}{Lemma}[section]
\newtheorem{corollary}{Corollary}[section]
\newtheorem{definition}{Definition}[section]
\newtheorem{remark}{Remark}[section]

\newtheorem{Thm}{Theorem}[section]
\newtheorem{Lem}[Thm]{Lemma}
\newtheorem{Prop}[Thm]{Proposition}

\theoremstyle{remark}
\newtheorem{Def}[Thm] {Definition}

\newcommand{\ka} {\kappa}

\newcommand{\N}{{\mathbb{N}}}

\long\def\begcom#1\endcom{}

\newcommand{\length}{\operatorname{\length}}

\def\length{\operatorname{length}}

\def\ln{\operatorname{ln}}

\newcommand{\bl} {\begin{lemma}}
	\newcommand{\el} {\end{lemma}}
\newcommand{\bt} {\begin{theorem}}
	\newcommand{\et} {\end{theorem}}

\newcommand{\bp}{\begin{proof}}
	\newcommand{\ep}{\end{proof}}
\newcommand  {\ee} {\end{equation}}
\newcommand  {\beq} {\begin{eqnarray*}}
	\newcommand  {\eeq} {\end{eqnarray*}}

\newcommand  {\bd} {\begin{definition}}
	\newcommand  {\ed} {\end{definition}}

\def\ep{\noindent{\hfill $\Box$}}

\begin{document}
	
\title{Bowen metric mean dimension formula for saturated sets}

\author{Yi Yuan}

\address{Yi Yuan, School of Mathematical Sciences,  Xihua University\\Chengdu 610000, People's Republic of China}
\email{19110180015@fudan.edu.cn}

\begin{abstract}
For dynamical systems with infinite topological entropy, the classical entropy fails to quantify their complexity effectively, while the metric mean dimension provides a natural extension in this context. In this paper, we study the complexity of saturated sets from the perspective of Bowen upper and lower metric mean dimensions. We show that if a dynamical system $(X,f)$ satisfies the $g$-almost product property, then for any compact connected non-empty subset $K$ of a set of the convex combination of finitely many invariant measures, the saturated set $G_K$ satisfies
$$
\overline{\operatorname{mdim}}^B_{M}\left(G_K, f, d\right)=\limsup_{\varepsilon \rightarrow 0} \frac{1}{|\ln \varepsilon|} \inf_{\mu \in K} \inf _{\operatorname{diam}(\xi)<\varepsilon} h_\mu(f, \xi),
$$
$$
\underline{\operatorname{mdim}}^B_{M}\left(G_K , f, d\right)=\liminf_{\varepsilon \rightarrow 0} \frac{1}{|\ln \varepsilon|} \inf_{\mu \in K} \inf_{\operatorname{diam}(\xi)<\varepsilon} h_\mu(f, \xi),
$$
where $\overline{\operatorname{mdim}}^B_{M}$ and $\underline{\operatorname{mdim}}^B_{M}$ denote  Bowen upper and lower metric mean dimensions of $f$ on $G_K$, respectively,  and $h_\mu(f,\xi)$ is the measure-theoretic entropy of the measure $\mu$ with respect to the partition $\xi$.
As an application, we give an abstract framework for multifractal analysis of general continuous functions, which  extends the prior work of Backes (2023, \textit{IEEE Trans. Inform. Theory}, 69, 5485–5496) and Liu (2024, \textit{J. Math. Anal. Appl.}, 534, 128043).
\end{abstract}
%\tableofcontents

\maketitle

\section{Introduction}
Let  $(X, d)$  be a compact metric space with Borel  $\sigma$-algebra  $\mathcal{B}(X)$, and let  $f: X \to X$  be a continuous map. Such a pair  $(X, f)$  is called a dynamical system. For a dynamical system $(X,f)$, let $\mathcal{M}(X)$, $\mathcal{M}_{f}(X)$, $\mathcal{M}^{e}_{f}(X)$ denote the space of probability measures, $f$-invariant, $f$-ergodic probability measures, respectively.
Let $\mathbb{Z},$ $\mathbb{N}$ and $\mathbb{N^{+}}$ denote the set of integers, non-negative integers and positive integers, respectively.

 In the study of dynamical systems, characterizing the complexity of point sets has always been a central topic, and
topological entropy is a powerful tool for quantifying such complexity. The concept of  topological entropy was  first introduced by Adler, Konheim, and McAndrew \cite{AKM1965} and later extended by Bowen \cite{Bowen1973} to non-compact sets.
Over the past decades, extensive research has been devoted to study the entropy of various orbit sets.
 For instance,
 the irregular sets \cite{Thompson2010,Dong2018irregular,Chen2023Lyapunov,Zhao2011asymptoticallyadditive} have been shown to possess full topological entropy, and a variational principle connecting  Bowen topological entropy of level sets and measure-theoretic entropy  has been established \cite{TV2003,Barreira2001Variationalprinciples}.  These studies are usually known as part of multifractal analysis associated with Birkhoff averages.

 In the theory of multifractal analysis, saturated sets have been particularly well studied.  A saturated set is defined as the collection of points whose empirical measures accumulate exactly on a given compact connected non-empty subset of invariant measures. More precisely, 
  Given $x\in X$,   the empirical measure of  $x$ is 
 $\mathcal{E}_{n}(x) : =\frac{1}{n}\sum_{j=0}^{n-1}\delta_{f^{j}x}$, 
 where $\delta_{x}$ denotes the Dirac measure at $x$.
 Let $V_f(x)$ be  the set of accumulation points of $\{\mathcal{E}_n(x)\}.$   It is well known that  $V_f(x)$  is a compact connected non-empty subset of  $\mathcal{M}_f(X)$, thus,  for any compact connected non-empty subset $K\subset \mathcal{ M}_f(X)$, we can define  the saturated  set of $K$ as follows $$G_K : =\left\{x \in X: V_f(x)=K\right\}. $$  
 when $K=\{\mu\}$ is a singleton,  $G_{K}$ is called the generic set of  $\mu$.  These sets provide a natural decomposition of the space $X$ according to the asymptotic behavior of empirical measures.

  Existing studies mainly focused on analyzing the  entropy formula for saturated sets.
 Bowen \cite{Bowen1973} showed that the topological entropy of the generic set of an ergodic measure $\mu$  coincides with the measure-theoretic entropy $h_{\mu}(f)$, i.e.,
 $$
 h_{\text {top }}^B\left(f, G_\mu\right)=h_\mu(f),
 $$ 
 where $h_{\text {top }}^B\left(f, G_\mu\right)$ denotes the  Bowen entropy of the set $G_\mu$.  
   Pfister and Sullivan \cite{PS2007} gave the entropy formula  
 $$h_{\text {top }}^B\left(f, G_K\right)=\inf\{ h_{\mu}(f): \mu \in K \},$$
 for  systems  satisfying uniform separation property and  $g$-almost product property, where $K$ is any compact connected non-empty subset of $\mathcal{M}_f(X)$.  Recently, Huang, Tian, and Wang \cite{HTW2019} considered the transitively saturated set $ G_K \cap \operatorname{Trans} $ to study the entropy of recurrence sets, where $\operatorname{Trans}$ denotes the set of transitive points.

However, in many highly complex systems, the topological entropy may fail to provide meaningful information. For example, let  $X$ be an infinite compact metric space and $\sigma: X^{\mathbb{N}} \to X^{\mathbb{N}} $ be the shift map defined by $(\sigma x)_k= x_{k+1}$. In this case, the topological entropy is infinite and thus  no longer serves as an effective label for classification.  Moreover, Yano \cite{Yano1980} proved that  $C^0$-generic dynamical systems have infinite topological entropy, indicating that this limitation occurs in a wide class of systems.  To obtain a new invariant capable of distinguishing maps with infinite entropy, Lindenstrauss and Weiss \cite{LW2000}  introduced the upper and lower metric mean dimensions, which is a natural extension of entropy, capturing the growth rate of information as the scale of observation—or equivalently, the diameter of partitions—shrinks, the precise definitions are given in Section \ref{subsection_metricmeandimension}.
Afterwards, 
%inspired by Bowen’s definition of entropy and Carath\'{e}odory-Pesin structures, 
Bowen upper and lower metric mean dimension of $f$ on a subset $Z\subseteq X$ (not necessarily compact or invariant) with respect to the metric $d$ were introduced in \cite{Cheng2021Uppermetricmean}. They are denoted by  $\overline{\operatorname{mdim}}^B_M(X, f, d)$ and $\underline{\operatorname{mdim}}^B_M(X, f, d)$, respectively, the precise definitions can refer to Section \ref{subsection_Bowenmetricmeandimension}.

  The theory of the (Bowen) metric mean dimension has made significant progress in recent years.  Existing research has mainly concerned the embedding problem \cite{Lindenstrauss1999Meandimension,LW2000}, rate distortion theory \cite{Velozo2017arxiv,Wang2021},  compression  theory  \cite{Gutman2019Newuniform ,Gutman2020Metricmean}, variational principles for metric mean dimension \cite{Tsukamoto2020, Wang2021, CRV2022}, and  Bowen metric mean dimension of irregular sets \cite{BR2023} and level sets \cite{LL,FO}.   The packing metric mean dimension formula for generic sets  has been established by Yang, Chen, and Zhou \cite{YCZ}. Motivated by the study of saturated sets from the perspective of topological entropy, a natural question arises: \textbf{Does  Bowen metric mean dimension formula still exist for saturated sets?}
  %Such a principle is crucial, on the one hand, it  carry more information besides infinite Bowen topological entropy, would overcome the limitation of topological entropy in quantifying the complexity of saturated sets with infinite entropy,on the other hand, ...
In this paper, we give an affirmative answer to this  question.  First, we give the following Bowen metric mean dimension formula for  generic set of an ergodic measure.

%The primary goal of this paper is to establish a metric mean dimension formula for saturated sets.

\begin{Thm}\label{Maintheorem_G_mu}
  Suppose $(X,f)$ is a dynamical system and $\mu \in \mathcal{ M}_f(X)$ is ergodic. Then 
	$$\overline{\operatorname{mdim}}_M^B(G_{\mu},f,d)=\limsup_{\varepsilon \rightarrow 0}\frac{1}{|\ln \varepsilon|}\inf_{ \operatorname{diam} (\xi)<\varepsilon}h_{\mu}(f,\xi),$$
	and 
	$$\underline{\operatorname{mdim}}_M^B(G_{\mu},f,d)=\liminf_{\varepsilon \rightarrow 0}\frac{1}{|\ln \varepsilon|}\inf_{ \operatorname{diam} (\xi)<\varepsilon}h_{\mu}(f,\xi),$$
	where $\operatorname{diam} (\xi)$ denotes the diameter of the partition $\xi$ and the infimum is taken over all finite measurable partitions of $X$ satisfying $\operatorname{diam}(\xi)<\varepsilon$.
\end{Thm}

Moreover, we investigate Bowen metric mean dimension of the saturated set $G_K$ for systems satisfying  $g$-almost  product  property (see Definition \ref{Def_galmost}), where  $K$  is the compact connected non-empty subset of a set of the convex combination of finitely many invariant measures.
Given $x\in X$,
the $\omega$-limit set of $x$ is given by $\omega_f(x) :  =\bigcap_{n=0}^{\infty} \overline{\bigcup_{k=n}^{\infty}\left\{f^k x\right\}}.$   The support of the measure $\mu \in \mathcal{M}_f(X)$ is defined as 
$$S_\mu : =\{x\in X:\mu(U)>0\ \text{for any neighborhood}\ U\ \text{of}\ x\},$$ and the measure center of the system is defined by  $C_f(X) =\overline{\bigcup_{\mu\in\mathcal{M}_f(X)}S_{\mu}}.$
For any \( m \in \mathbb{N}^+ \) and \( \{\nu_i\}_{i=1}^m \subseteq \mathcal{M}_{f}(X) \), \( \operatorname{conv}\{\nu_i\}_{i=1}^m \)  denotes the set of convex combination of \( \{\nu_i\}_{i=1}^m \), i.e.,  
$$
\operatorname{conv}\left\{\nu_i\right\}_{i=1}^m : = \left\{\sum_{i=1}^m t_i \nu_i : t_i \in [0,1], 1 \leq i \leq m, \sum_{i=1}^m t_i = 1\right\}.
$$
We define the refined saturated set associated with 
$K$ as
$$G^C_K : = G_K \cap \{x \in X :  C_f(X) \subset \omega_f(x)\}.$$

\begin{Thm}\label{MainTheorem}
	 Suppose $(X,f)$ is a dynamical system  with the $g$-almost product property. For any $m \in \mathbb{N}$, $ \{\mu_i\}_{i=1}^m\subset\mathcal{M}_f(X)$ and any  compact connected non-empty subset  $K\subset \operatorname{conv} \{\mu_i\}_{i=1}^m$, one has
	\begin{eqnarray*}
	\overline{\operatorname{mdim}}^B_{M}\left(G^C_K, f, d\right) = \limsup _{\varepsilon \rightarrow 0} \frac{1}{|\ln \varepsilon|} \inf_{\mu \in K} \inf _{\operatorname{diam}(\xi)<\varepsilon} h_\mu(f, \xi),
\end{eqnarray*}
	and
		\begin{eqnarray*}
		\underline{\operatorname{mdim}}^B_{M}\left(G^C_K, f, d\right) = \liminf_{\varepsilon \rightarrow 0} \frac{1}{|\ln \varepsilon|} \inf_{\mu \in K} \inf _{\operatorname{diam}(\xi)<\varepsilon} h_\mu(f, \xi),
	\end{eqnarray*}
	where $\operatorname{diam} (\xi)$ denotes the diameter of the partition $\xi$ and the infimum is taken over all finite measurable partitions of $X$ satisfying $\operatorname{diam}(\xi)<\varepsilon$.
\end{Thm}
\begin{remark}
	According to the approach of Pfister and Sullivan \cite{PS2007},  the uniform separation property is required to establish the entropy formula for saturated sets of any  compact connected non-empty subset $K \subset \mathcal{M}_f(X)$.
	When considering any  compact connected non-empty subset of a set of convex combination of finite measures, the uniform separation property is not necessary.
	 Since for a dynamical system with the uniform separation property,  the topological entropy is finite \cite[Page 938]{PS2007}, then the metric mean dimension for such systems is zero,  therefore,  this question is meaningless. Thus,  it remains a problem to establish a Bowen metric mean dimension formula for saturated sets of arbitrary compact connected subsets. 
\end{remark}

\begin{remark}
	Further, if there is an invariant measure with full support, then  $G_K^C=G_K\cap \operatorname{Trans}$, where $\operatorname{Trans}:=\{x \in X : X= \omega_f(x)   \} $, and $$\overline{\operatorname{mdim}}^B_{M}\left(G_K\cap \operatorname{Trans} , f, d\right)=\limsup _{\varepsilon \rightarrow 0} \frac{1}{|\ln \varepsilon|} \inf_{\mu \in K} \inf _{\operatorname{diam}(\xi)<\varepsilon} h_\mu(f, \xi),$$
	and  $$\underline{\operatorname{mdim}}^B_{M}\left(G_K \cap \operatorname{Trans} , f, d\right)=\liminf_{\varepsilon \rightarrow 0} \frac{1}{|\ln \varepsilon|} \inf_{\mu \in K} \inf _{\operatorname{diam} (\xi)<\varepsilon} h_\mu(f, \xi).$$
\end{remark}

Directly from Theorem \ref{MainTheorem}, we obtain the following corollary.

\begin{corollary}\label{cor-maintheorem}
 Suppose $(X,f)$ is a dynamical system  with the $g$-almost product property. Then for any $\mu \in \mathcal{M}_f(X)$,  we have
	$$
	\overline{\operatorname{mdim}}^B_{M}\left(G_\mu, f, d\right)=\limsup _{\varepsilon \rightarrow 0} \frac{1}{|\ln \varepsilon|} \inf _{\operatorname{diam} (\xi)<\varepsilon} h_\mu(f, \xi),
	$$
	and
	$$\underline{\operatorname{mdim}}^B_{M}\left(G_{\mu}, f, d\right)=\liminf_{\varepsilon \rightarrow 0} \frac{1}{|\ln \varepsilon|}  \inf _{\operatorname{diam} (\xi)<\varepsilon} h_\mu(f, \xi).$$
\end{corollary}

Saturated sets can be used to describe convergence properties of Birkhoff averages with respect to every continuous function. On the one hand, by weak∗ convergence, $V_f(x)$ is a singleton if and only if the Birkhoff averages of every continuous function converge at $x$. On the other hand, if the accumulation set $V_f(x)$ contains invariant measures $\mu_1, \mu_2$ that can be distinguished by a continuous function $\varphi$ (i.e., $\int \varphi d \mu_1 \neq \int \varphi  d \mu_2$), then $x$ is a Birkhoff irregular point associated to $\varphi$. Therefore, a thorough understanding of saturated sets in the space of probability measures enables us to derive key results  in the multifractal analysis of Birkhoff averages associated with continuous function.

The multifractal analysis of Birkhoff averages is to decompose the space $X$ into subsets of points exhibiting similar dynamical behavior,  and to describe the size of each subsets from the geometrical or topological viewpoint.
 Specifically, for a continuous function $\varphi: X \rightarrow \mathbb{R}$, 
 let $$L_\varphi : =\left[\inf _{\mu \in \mathcal{ M}_f(X)} \int \varphi d\mu, \sup _{\mu \in  \mathcal{ M}_f(X)} \int \varphi d\mu \right].$$
The space  $X$  naturally admits a multifractal decomposition $X=\cup_{a \in L_{\varphi}} R_{\varphi}(a) \cup I_{\varphi},$
where
$$
R_{\varphi}(a) : =\left\{x \in X: \lim _{n \rightarrow \infty} \frac{1}{n} \sum_{i=0}^{n-1} \varphi\left(f^i(x)\right)=a\right\} $$
and
$$I_{\varphi} : =\left\{x \in X: \lim _{n \rightarrow \infty} \frac{1}{n} \sum_{i=0}^{n-1} \varphi\left(f^i(x)\right) \textit{ does not exist }\right\}.
$$
The sets  $R_{\varphi}(a)$  and  $I_{\varphi}$  have been extensively studied from the perspective of entropy \cite{TV1999,TV2003,Thompson2009,Thompson2010}. Recently, their investigation from the viewpoint of Bowen metric mean dimension has attracted much attention \cite{BR2023, LL, FO}, these results are mainly obtained for systems satisfying the specification or shadowing property.

In this paper, we present another application of Theorem \ref{MainTheorem} to multifractal analysis,  extending the results of Backes \cite{BR2023} and Liu \cite{LL}. In fact,  we  establish an abstract framework for multifractal analysis of general functions in Section \ref{Section_5}.
 For any  constant $a\in L_{\varphi},$
define$$R^C_{\varphi}(a) :  =R_{\varphi}(a)\cap  \{x\in X: C_f(X) \subset \omega_f(x)\}, \  I^C_{\varphi}:=I_{\varphi}\cap   \{x\in X: C_f(X) \subset \omega_f(x)\},$$ 
 and $$H_{\varphi}(a,\varepsilon) : =  \frac{1}{|\ln \varepsilon|} \sup_{\mu \in \mathcal{M}_f(X,\varphi,a)} \inf _{\operatorname{diam}(\xi)<\varepsilon} h_\mu(f, \xi),$$
  where $\mathcal{M}_f(X,\varphi, a) : = \{ \mu \in \mathcal{M}_f( X) : \int \varphi d\mu=a\}$.
  \begin{corollary}\label{cor_multifractual}
  Suppose $(X,f)$ is a dynamical system  with the $g$-almost product property  and $\varphi:  X \rightarrow \mathbb{R}$ is a continuous function. 
  	\begin{enumerate}
  		\item[(1)] For any real number $a \in L_\varphi$, the set $R_\varphi^C(a)$ is not empty and
  		$$
  		\overline{\operatorname{mdim}}^B_M\left(R_\varphi(a),f,d\right)=	\overline{\operatorname{mdim}}^B_M\left(R_\varphi^C(a) ,f,d\right)=\limsup_{\varepsilon \rightarrow 0}H_{\varphi}(a,\varepsilon).
  		$$
  		\item[(2)] $\overline{\operatorname{mdim}}^B_M\left(X,f,d\right) = \overline{\operatorname{mdim}}^B_M\left(\cup_{a \in L_{\varphi}} R_{\varphi}(a),f,d\right)$.
  		\item[(3)] If $I_{\varphi}\neq \emptyset$, then
  		$I_\alpha^C\neq \emptyset$. Moreover,
  		 $$\overline{\operatorname{mdim}}^B_M\left(I_\alpha,f,d\right)=	\overline{\operatorname{mdim}}^B_M\left(I^C_\alpha ,f,d\right)=	\overline{\operatorname{mdim}}^B_M\left(X,f,d\right).$$
  	\end{enumerate}
    \end{corollary}
\begin{remark}
   Similar results to Corollary \ref{cor_multifractual} also hold for the lower metric mean dimension by replacing $\overline{\operatorname{mdim}}_M^B$ with $\underline{\operatorname{mdim}}_M^B$ and $\limsup_{\varepsilon \rightarrow 0}$ with $\liminf_{\varepsilon \rightarrow 0}$. For brevity, the detailed statement is omitted.
\end{remark}
\begin{remark}
	The part for $I_{\varphi}$ has been investigated by Liu and Liu \cite{LL} for the systems with  almost weak specification property. 	The part for $R_{\varphi}(a)$ has been studied by  Backes and Rodrigue \cite{BR2023} for the systems with specification property.  
	\end{remark}

\textbf{Organization of this paper.}
In preparation for proving the main results, we recall some notations and definitions in Section \ref{Section_2}.
The proofs of Theorems \ref{Maintheorem_G_mu} and \ref{MainTheorem} are presented in Section \ref{Section_3}.
 In Section \ref{Section_5}, we give an application to multifractal analysis of general continuous functions.

\section{Preliminaries}\label{Section_2}

Let $(X, f)$ be a dynamical system, $d$ is the metric on $X$. Given $\varepsilon>0,$
we denote by 
$$B(x, \varepsilon): =\left\{y \in X: d(x, y)<\varepsilon\right\}$$
the $\varepsilon$-ball of $x$.
For  $n \in \mathbb{N}$,  the Bowen metric $d^f_n$ on $X$ is 
$$d^f_n(x, y) : =\max _{0 \leq i \leq n-1}\left\{d \left(f^i x, f^i y\right)\right\}.$$ 
%It is clear that $d^f_n$ is a metric generating the same topology as $d$ for each $n \in \mathbb{N}$. 
We denote by
$$B_n(x, \varepsilon,f): =\left\{y \in X: d^f_n(x, y)<\varepsilon\right\}$$
the $(n, \varepsilon)$-ball of $x$.
When not specified, we abbreviate $d_n^f(x,y)$, $B(x, \varepsilon)$ and $ B_n(x,\varepsilon,f)$ as $d_n(x,y)$, $B(x, \varepsilon)$ and $ B_n(x,\varepsilon)$, respectively.

Let $C(X)$ denote the set of continuous functions on $X$ and $||\phi||=\max \{|\phi(x)| : x \in X\}$ be the  norm  of $\phi \in C(X)$.
We set
$$
\langle \phi, \mu \rangle : =\int_{X} \phi d \mu.
$$
There exists a sequence $\{\phi_k\}_{k \in \mathbb{N}}$  that is a dense subset of $C(X)$, and $0 \leq \phi_{k}(x) \leq 1$ for all $k\in\mathbb{N}$ and $x \in X$. The metric for the weak$^{*}$ topology on $\mathcal{M}(X)$ is defined as
$$
\rho(\mu, \nu) : =\sum_{k \geq 1} 2^{-k}\left|\left\langle \phi_{k}, \mu\right\rangle-\left\langle \phi_{k}, \nu \right\rangle\right|,
$$
obviously,
\begin{equation}\label{eq-d}
\rho(\mu, \nu) \leq \sum_{k \geq  1}2^{-k+1} \leq 2.
\end{equation}
We denote by
$$B(\mu, \varepsilon) : =\left\{\nu\in \mathcal{M}(X) : \rho \left(\mu, \nu \right)<\varepsilon\right\} $$ 
the $\varepsilon$-ball of $\mu$.

\subsection{$g$-almost product property}\label{Section_g_almostproductproperty}
The $g$-almost product property was first introduced by  Pfister and Sullivan,   which is weaker than specification property \cite[Proposition 2.1]{PS2007}.  Let $g: \mathbb{N} \rightarrow \mathbb{N}$ be a given non-decreasing unbounded map with 
$$
g(n)<n  \text { and }  \lim _{n \rightarrow \infty} \frac{g(n)}{n}=0,
$$
such function $g$ is called a blowup function.
Let $x \in X$ and $\varepsilon>0$, the $g$-blowup of $B_n(x, \varepsilon)$ is the closed set
$$B_n(g; x, \varepsilon) : =\left\{y \in X: \exists \Lambda \subset \Lambda_n,\sharp (\Lambda_n \backslash \Lambda ) \leq g(n),\ \max \left\{d\left(f^j x, f^j y\right): j \in \Lambda\right\} \leq \varepsilon\right\}.$$

\begin{Def}\label{Def_galmost}
A dynamical system $(X, f)$ is said to have the $g$-almost product property with blowup function $g$, if there exists a non-increasing function $m: \mathbb{R}^{+} \rightarrow \mathbb{N}$ such that for any $k \in \mathbb{N}$, any $x_1 , \ldots, x_k \in X$, any positive number $\varepsilon_1, \ldots, \varepsilon_k$, and any integers $n_1 \geq m\left(\varepsilon_1\right), \ldots, n_k \geq m\left(\varepsilon_k\right)$, 
	$$
	\bigcap_{j=1}^k f^{-M_{j-1}} B_{n_j}\left(g; x_j, \varepsilon_j\right) \neq \emptyset,
	$$
	where $M_0 =0, M_i =n_1+\cdots+n_i, i=1, \ldots, k-1$.
\end{Def}

A point $x\in X$ is almost periodic, if for every open neighborhood $U$ of $x$, there exists $N \in \N^+$ such that for every $n\in\N$ there is $n\leq k\leq n+N$ such that $f^kx \in U$. We denote the set of almost periodic points by $AP(X)$. 
The following proposition reveals that the almost periodic points play an important role in studying the measure center for the systems with  $g$-almost product property.
\begin{Prop}\label{lemma-APdense}\cite[Proposition 2.11]{HTZ2023}
	Suppose  $(X,f)$ is  a dynamical system with $g$-almost product property. Then the almost periodic set $AP(X)$ is dense in $C_f(X)$.
\end{Prop}

\subsection{Measure-theoretic entropy }
Let $\mu \in \mathcal{M}_f(X)$. We say that $\xi=\left\{C_1, \ldots, C_k\right\}$ is a measurable partition of $X$ if each $C_i$ is a measurable set, $\mu\left(X \backslash \cup_{i=1}^k C_i\right)=0$ and $\mu\left(C_i \cap C_j\right)=0$ for every $i \neq j$. The entropy of $\xi$ with respect to $\mu$ is given by
$$
H_\mu(\xi) : =-\sum_{i=1}^k \mu\left(C_i\right) \ln \mu\left(C_i\right).
$$
The entropy of $\xi$ given a measurable partition $\zeta=\{A_1,\cdots,A_p\}$ is the number
\begin{equation}\label{eq_conditional}
	H_{\mu}(\xi | \zeta) :  = -\sum_{i, j} \mu \left(C_i \cap A_j\right) \ln \frac{\mu \left(C_i \cap A_j\right)}{\mu \left(A_j\right)}.
\end{equation}
Let $\xi^n =\bigvee_{j=0}^{n-1} f^{-j} \xi$, the metric entropy of $f$ with respect to $\xi$ is given by
$$
h_\mu(f, \xi) : =\lim _{n \rightarrow+\infty} \frac{1}{n} H_\mu\left(\xi^n\right) .
$$
Fix $\varepsilon>0$, one has $\inf_{\operatorname{diam}(\xi)<\varepsilon}h_{\mu}(f,\xi)<+\infty$ (see  \cite[Theorem 2]{Shi2022}). 
The metric entropy of $f$ respect to $\mu$ is 
$$
h_\mu(f) : =\sup _{\xi} h_\mu(f, \xi),
$$
where the supremum is taken over all finite measurable partitions $\xi$ of $X$.  Here we recall some properties of the measure-theoretic entropy.
\begin{Prop}\cite{water2001}\label{Prop_propertyofh}
	 Let  $(X, \mathcal{ B},\mu )$ be the probability space and $f$ be a measure preserving transformation. If $\xi,
	 \eta,\gamma$ are measurable partitions of $X$, then
	\item[(1)] 
	$ H_\mu(\eta \vee \gamma \mid \xi)  \leq H_\mu(\eta \mid \xi)+H_{\mu}(\gamma \mid \xi).$
	\item[(2)]
	$ H_\mu\left(f^{-1} \eta \mid f^{-1} \xi\right)  =H_\mu(\eta \mid \xi).$
	\item[(3)]
	$ h_\mu(f, \eta)  \leq h_\mu(f, \xi)+H_\mu(\eta \mid \xi).$
	\item[(4)] $h_{\mu}\left(f^k, \bigvee_{i=0}^{k-1} f^{-i} \xi \right)=k h_{\mu}(f, \xi)$ if $k>0$.

\end{Prop}

%\subsection{Katok‘s entropy}
%The Lebesgue number of an open cover $\mathcal{U}$ of $X$, denoted by $\operatorname{Leb}(\mathcal{U})$, is the largest number $\varepsilon>0$ such that every open ball of radius $\varepsilon$ is contained in an element of $\mathcal{U}$. Denote %$$\operatorname{diam}(\mathcal{U})=\max \left\{\operatorname{diam}\left(U_i\right): U_i \in \mathcal{U}\right\}.$$
%Given  $\mu \in \mathcal{ M}_f^e(M)$, for $\delta \in(0,1), n \in \mathbb{N}$ and $\varepsilon>0$,   let $N_\mu^\delta(n, \varepsilon)$ be the smallest number of  $(n, \varepsilon)$-balls whose union has $\mu$-measure larger than $\delta$. 
%Let$\tilde{N}_\mu^\delta(n, \varepsilon)$ be the smallest number of sets with diameter at most $\varepsilon$ in the metric $d_n$ whose union has $\mu$-measure larger than $\delta$.  The following lemma reveals the relation between $N_\mu^\delta(n, \varepsilon)$ and $\tilde{N}_\mu^\delta(n, \varepsilon)$.
%Define $$ \underline{h}_\mu(f, \varepsilon, \delta)=\liminf _{n \rightarrow \infty} \frac{\ln N_\mu^\delta(n, \varepsilon)}{n} \text { and } \bar{h}_\mu(f, \varepsilon, \delta)=\limsup _{n \rightarrow \infty} \frac{\ln N_\mu^\delta(n, \varepsilon)}{n}. 
%$$ The following  entropy formula is given by the Katok. 
%\begin{Thm}\cite[Theorem 1.1]{Katok 1980}	Let $(X, f)$ is a dynamical system and $\mu \in \mathcal{ M}_f^e(X)$. Then for any $\delta \in(0,1)$, $$ h_\mu(f)=\lim _{\varepsilon \rightarrow 0} \underline{h}_\mu(f, \varepsilon, \delta)=\lim _{\varepsilon \rightarrow 0} \bar{h}_\mu(f, \varepsilon, \delta). $$ \end{Thm}

\subsection{Pfister and Sullivan’s entropy}
For $\delta>0$ and $\varepsilon>0$, two points $x, y\in X$ are $(\delta, n, \varepsilon)$-separated if
$$
\sharp \left\{0 \leq j \leq n-1: d\left(f^j x, f^j y\right)>\varepsilon\right\} \geq \delta n.
$$
A subset $E\subseteq X$ is $(\delta,n,\varepsilon)$-separated  if any pair of distinct points in $E$ are $(\delta,n,\varepsilon)$-separated. A set $E \subset X$ is $(n, \varepsilon)$-separated if $d_n(x, y) > \varepsilon$ for every $x \neq y \in E.$
Let $F\subseteq \mathcal{M}_f(X)$ be a neighborhood and $n\in\N$,  define
$$X_{n,  F} : =\{x\in X:\mathcal{E}_{n}(x)\in F\}. $$
Let $N(F, \delta, n,\varepsilon)$ denote the maximal cardinality of the $(\delta ,n,\varepsilon)$-separated subset of $X_{n,F}$ and
$N(F,n,\varepsilon)$ be the maximal cardinality of the $(n,\varepsilon)$-separated subset of $X_{n,F}.$ Define
$$ \overline{PS}(\mu, \varepsilon) : =\inf _{F \ni \mu} \limsup _{n \rightarrow \infty} \frac{1}{n} \ln N(F,n,\varepsilon),\ \underline{PS}(\mu, \varepsilon) : =\inf_{F \ni \mu} \liminf _{n \rightarrow \infty} \frac{1}{n} \ln N(F,n,\varepsilon),$$
and
$$ \overline{PS}(\mu,\delta,\varepsilon) : =\inf_{F \ni \mu}
\limsup_{n \rightarrow \infty} \frac{1}{n} \ln N(F, \delta, n, \varepsilon),\ \underline{PS}(\mu,\delta,\varepsilon) : =\inf_{F \ni \mu}\liminf_{n \rightarrow \infty} \frac{1}{n} \ln N(F, \delta, n, \varepsilon).$$
 \begin{Thm}\cite[Corollary 3.2]{PS2007}\label{Thm_entropyformula_PS}
 	 Let $(X, f)$ be a dynamical system and $\mu\in \mathcal{ M}_f^e(X)$. Then 
 	$$
 	\begin{aligned}
 	h_{\mu}(f) &=\lim _{\varepsilon \rightarrow 0} \lim _{\delta \rightarrow 0}\overline{PS}(\mu,\delta,\varepsilon)
 	=\lim _{\varepsilon \rightarrow 0} \lim_{\delta \rightarrow 0}  \underline{PS}(\mu,\delta,\varepsilon)\\
 	&= \lim _{\varepsilon \rightarrow 0} \overline{PS}(\mu, \varepsilon)= \lim _{\varepsilon \rightarrow 0} \underline{PS}(\mu, \varepsilon).
 	\end{aligned}
 	$$
 \end{Thm}

\subsection{Metric mean dimension}\label{subsection_metricmeandimension}
Let $s(f, n, \varepsilon)$ denote the maximal cardinality of all $(n, \varepsilon)$-separated subsets of $X$, which is  finite since $X$ is compact.
The upper and lower metric mean dimension of $f$ with respect to $d$ are defined respectively by
\begin{equation}\label{def-uppermetricmeandimension-0}
	\overline{\operatorname{mdim}}_{M}(X, f, d) : =\limsup _{\varepsilon \rightarrow 0} \frac{h(f, \varepsilon)}{|\ln \varepsilon|},
\end{equation}
and
\begin{equation}\label{def-lowermetricmeandimension-0}
	\underline{\operatorname{mdim}}_{M}(X, f, d) : =\liminf _{\varepsilon \rightarrow 0} \frac{h(f, \varepsilon)}{|\ln \varepsilon|},
\end{equation}
where $h(f, \varepsilon)=\limsup _{n \rightarrow \infty} \frac{1}{n} \ln s(f, n, \varepsilon).$
In the case when $\underline{\operatorname{mdim}}_{M}(X, f, d) = \overline{\operatorname{mdim}}_{M}(X, f, d)$ this common value is called the metric mean dimension of $f$ with respect to $d$ and is denoted simply by $\operatorname{mdim}_{M}(X, f, d)$.
The topological entropy of $f$ is 
$ h_{top}(f)=\lim _{\varepsilon \rightarrow 0} h(f, \varepsilon). $
Consequently, $$\overline{\operatorname{mdim}}_{M}(X, f, d)=\underline{\operatorname{mdim}}_{M}(X, f, d)=0,$$ whenever the topological entropy of $f$ is finite.

The following variational principle for the metric mean dimension was established by Gutman and Śpiewak.

\begin{Thm}\cite[Theorem 3.1]{GS2021}\label{Thm_variationprinciple}
	Let $(X, d, f)$ be a dynamical system. Then
	$$
	\overline{\operatorname{mdim}}_M(X, f, d)=\limsup _{\varepsilon \rightarrow 0} \frac{1}{|\ln \varepsilon|} \sup _{\mu \in \mathcal{M}_f(X)} \inf_{\operatorname{diam}(\xi)<\varepsilon} h_\mu(f, \xi),
 	$$
	and
	$$
	\underline{\operatorname{mdim}}_{M}(X, f, d)=\liminf _{\varepsilon \rightarrow 0} \frac{1}{|\ln \varepsilon|} \sup_{\mu \in \mathcal{M}_f(X)} \inf_{\operatorname{diam}(\xi)<\varepsilon} h_\mu(f, \xi).
	$$
\end{Thm}

\subsection{Bowen metric mean dimension}\label{subsection_Bowenmetricmeandimension}
Given a non-empty set $Z \subset X$, let
$$
m(Z, s, N, \varepsilon,f) : =\inf _{\Gamma}\left\{\sum_{i \in I} \exp \left(-s n_i\right)\right\},
$$
where the infimum is taken over all finite or countable collection $\Gamma(Z)=\left\{B_{n_i}\left(x_i, \varepsilon,f \right)\right\}_{i \in I}$ such that $Z \subset$ $\cup_{i \in I} B_{n_i}\left(x_i, \varepsilon,f \right)$ and $\min \left\{n_i: i \in I\right\} \geq N$. Note that $m(Z, s, N, \varepsilon,f)$ does not decrease as $N$ increases,  therefore,  the following limit exists
$$
m(Z, s, \varepsilon,f) : =\lim _{N \rightarrow \infty} m(Z, s, N, \varepsilon,f) .
$$
There exists a  number $s_0 \in[0,+\infty)$ such that $m(Z, s, \varepsilon,f )=0$ for every $s>s_0$ and $m(Z, s, \varepsilon,f)=+\infty$ for every $s<s_0$\cite[Proposition 1.2]{Pesin1997}. Define
$$ h_{\text {top }}^B(Z, f, \varepsilon) : =\inf \{s: m(Z, s, \varepsilon,f)=0\}=\sup \{s: m(Z, s, \varepsilon, f)=+\infty\}.$$
Note that $m(Z, h_{\text {top }}^B(Z, f, \varepsilon), \varepsilon,f)$  may take the value $+\infty, 0$ or some positive finite number.  Bowen topological entropy is defined by 
$$h_{\text {top }}^B(Z, f) : =\lim _{\varepsilon \rightarrow 0} h_{\text {top }}^B(Z, f, \varepsilon).$$
  Bowen upper and lower
metric mean dimension of $f$ on $Z$ with respect to $d$ are defined by
\begin{equation}\label{Def-uppermetricmeandimension}
	\overline{\operatorname{mdim}}_{M}^B(Z, f, d) : =\limsup _{\varepsilon \rightarrow 0} \frac{h_{\text {top }}^B(Z, f, \varepsilon)}{|\ln \varepsilon|},
\end{equation}
and
\begin{equation}\label{Def-lowmetricmeandimension}
	\underline{\operatorname{mdim}}_{M}^B(Z, f, d) : =\liminf _{\varepsilon \rightarrow 0} \frac{h_{\text {top }}^B(Z, f, \varepsilon)}{|\ln \varepsilon|},
\end{equation}
respectively. 
In the case when $Z = X$, the metric mean dimension and Bowen metric mean dimension given above actually coincide.
Here we give some basic properties of  Bowen metric mean dimension.
\begin{Prop}\label{Prop_propertyofdim}
	(1) If $Z_1 \subseteq Z_2$ are non-empty, then
	$$
	\overline{\operatorname{mdim}}_{M}^B\left(Z_1, f, d\right) \leq \overline{\operatorname{mdim}}_{M}^B\left(Z_2, f, d\right)\ \text{and} \ \underline{\operatorname{mdim}}_{M}^B\left(Z_1, f, d\right) \leq \underline{\operatorname{mdim}}_{M}^B\left(Z_2, f, d\right) .
	$$
	(2) For any $\varepsilon>0$, any $m\in \N$ and  any $Z\subseteq X$, one has
	$$ h_{top}^B(Z, f^m,\varepsilon)\leq  m h_{top}^B(Z, f,\varepsilon). $$ 
\end{Prop}
\begin{proof}
 The property (1) follows directly from the fact that $h_{top}^B(Z,f,\varepsilon)$ is a dimension characteristic \cite[Theorem 1.1]{Pesin1997}.
	
	(2) Fix  $\varepsilon>0$ and  $N\in \N$,
    $\Gamma_{N, f}$ and $\Gamma_{N,  f^m}$  denote  the collection of all finite or countable covers of $Z$ by sets of the form $\{B_n(x, \varepsilon,f)\}$ and  $\{B_n(x, \varepsilon,f^m)\}$ with $n \geq N$, respectively. 
	Let  $\{ B_{n_i}(x_i,\varepsilon,f) \}_{i\in I} \in \Gamma_{Nm,f}$ be the countable $(n_i,\varepsilon)$ Bowen ball collection of $(X,f)$  with  $ Z \subseteq  \cup_{i\in I}B_{n_i}(x_i,\varepsilon,f)$ with $\min\{n_i\}_{i\in I} > Nm$.
	For each $i\in \N$, let $n_i -1 = p_i m +q_i$, where $ 0 \leq q_i \leq m-1$ and $p_i \geq N$. Since $$ \cap_{j=0}^{n_i-1} f^{-j}B(f^j x_i,\varepsilon) \subseteq  \cap_{j=0}^{p_i} f^{-j m}B(f^{jm} x_i,\varepsilon),$$
 it follows that $$B_{n_i}(x_i,\varepsilon,f)  \subseteq B_{p_i+1}(x_i,\varepsilon,f^m).$$ Thus, the collection $\{B_{p_i+1}(x_i,\varepsilon,f^m) \}_{i\in I}$ also is the  open cover of $Z$, that is, $$\{B_{p_i+1}(x_i,\varepsilon,f^m) \}_{i\in I} \in \Gamma_{N,f^m}.$$
	Note that
	$$\frac{n_i}{m}=p_i+\frac{q_i+1}{m} \leq p_i +1,$$
	let $\alpha=h_{top}^B(Z,f^m,\varepsilon)$,
	then
	$$
	\begin{aligned}
		m\left(Z, \frac{\alpha}{m},m N,  \varepsilon,  f\right) & =\inf _{\Gamma_{mN,f}}\left\{\sum_{i\in I} \exp \left(-\frac{\alpha}{m}  n_i \right)\right\} \\
		& \geq \inf _{ \Gamma_{N,f^m} } \left\{\sum_{i\in I} \exp \left(-\alpha( p_i +1)\right)\right\} \\
		& =m(Z, \alpha, N,  \varepsilon,  f^m) .
	\end{aligned}
	$$
	Let $N \rightarrow \infty$, we have
	$$
	m\left(Z, \frac{\alpha}{m}, \varepsilon, f \right) \geq m(Z, \alpha, \varepsilon, f^m) .
	$$
	So,
	$$
	h_{\text {top }}^B(Z, f, \varepsilon) \geq \frac{ h_{\text {top }}^B(Z, f^m, \varepsilon)}{m}.
	$$
\end{proof}

%{Lem}\cite[Lemma 3.4]{HTW2019} \label{lemma-ergodicdensity}
%	Suppose that $(X,f)$ has $g$-almost product property. Then ergodic measures supported on minimal sets are dense in $\mathcal{M}_f(X)$.
%\end{Lem}

\section{Proofs of Theorem \ref{Maintheorem_G_mu} and Theorem \ref{MainTheorem}}\label{Section_3}

\subsection{Some lemmas}

\begin{Lem}\label{prop_infequal}
	Soppose $(X, f)$ is a dynamical system and $\xi$ is a finite measurable partition of $X$.  For any $\mu \in \mathcal{M}_f(X)$ and any $\varepsilon>0$, one has  $$\inf_{\operatorname{diam}( \xi )< \varepsilon,\mu(\partial \xi)=0 }h_{\mu}(f,\xi)=\inf_{\operatorname{diam} (\xi )< \varepsilon}h_{\mu}(f,\xi).$$
\end{Lem}
\begin{proof}
	Obviously, $\inf_{\operatorname{diam}( \xi )< \varepsilon,\ \mu(\partial \xi)=0 }h_{\mu}(f,\xi)\geq \inf_{\operatorname{diam} (\xi )< \varepsilon}h_{\mu}(f,\xi)$. So we only need to prove the reverse inequality.
	Fix $\varepsilon>0$, let $\mathcal{P}=\{P_1,P_2,\cdots, P_{n}\}$ be a finite partition with $\operatorname{diam}(\mathcal{P})<\varepsilon$ and $\mu(\partial \mathcal{P})=0$ and  $\mathcal{A}=\{A_1,\cdots,A_{k}\}$ be a partition of $X$ with $\operatorname{diam}(\mathcal{A})<\varepsilon$,  
	\cite[Lemma 8.5]{water2001} guarantees the existence of such $\mathcal{P}$ and $\mathcal{A}$.
	Let $\phi(x)=-x\ln x: [0,1] \rightarrow [0,1]$, for any $\eta>0$, there exists $\delta>0$ such that $\phi(x) <\eta$, if
	\begin{equation}\label{eq_xlnx1}
		0<x<\frac{k\delta}{\min_{1\leq j\leq k}\{\mu(A_j)\}},
	\end{equation}
or
	\begin{equation}\label{eq_xlnx2}
	0<1-x< \frac{\delta}{\min_{1\leq j\leq k}\{\mu(A_j)\}}. 
	\end{equation}
	Since $\mu$ is a regular measure, there exists a closed subset $B_i \subset A_i$ satisfying $\mu(A_i\backslash B_i)<\frac{\delta}{2}$. Thus, $\mu(A_i \Delta B_i) <\frac{\delta}{2}$ for all $i=1,\cdots,k$.  
	Denote $B_0=X\backslash \cup_{i=1}^kB_i$. Then we obtain a partition $\mathcal{B}=\{B_0,B_1,\cdots,B_k\}$.  
	Let $\alpha=\min_{i,j\neq 0,i\neq j}\{ d(B_i,B_j)\}>0$.  
	Since $\mu$ is regular, we can choose $B_i\subset U_i\subset \overline{U_i}\subset B(B_i,\alpha/2)$ such that $\mu(U_i\Delta B_i)<\frac{\delta}{2}$ and $\operatorname{diam}(U_i) <\varepsilon$, $i=1,\cdots,k$,  
	where $B(B_i,\alpha)=\{y\in X: d(B_i,y)<\alpha\}$.  
	Fix $i\in \mathbb{N}$,  let $\tau=\operatorname{dist}(B_i,X\backslash U_i)>0$.  
	For any $n>\frac{1}{\tau}$, there are at most $n$ balls $B(B_i,t)$ ($t<\tau$) with $\mu(\partial B(B_i,t)) \geq \frac{1}{n}$.  
	Then we can choose $B_i\subset V_i \subset \overline{V_i} \subset U_i$ such that $\mu(\partial V_i)=0$, $\mu(V_i\Delta B_i)<\frac{\delta}{2}$, and $\operatorname{diam}(V_i) <\varepsilon$ for $i=1,\cdots,k$.  
	Let $C_i=\overline{V_i}$ for $i=1,\cdots,k$ and $C_0=X\backslash \cup_{i=1}^k C_i$.  
	Then there is a partition $\mathcal{C}=\{C_0,C_1,\cdots,C_k\}$ satisfying  
	$$\mu(C_i \Delta A_i)\leq \mu(B_i \Delta U_i)+\mu(B_i\Delta A_i) <\delta, \quad \mu(\partial C_i)=0, \ i=1,\cdots,k,$$  
	and  
	$$\mu(C_0)= \mu\left( \cup_{j=1}^k A_j  \backslash \cup_{i=1}^k C_i \right) \leq \mu\left( \cup_{i=1}^k (A_i \Delta C_i) \right)<k \delta, \quad \mu(\partial C_0)=0.$$
	Let $\mathcal{C}_0|_{\mathcal{P}}=\{P_1\cap C_0, P_2\cap C_0,\cdots, P_{n}\cap C_0\}$.  
	Then $\operatorname{diam}(P_j\cap C_0)\leq \operatorname{diam}(P_j)<\varepsilon$ and $\mu(\partial (P_j\cap C_0))=0$.  
	Thus, we find a finite partition  
	$$\mathcal{D}= \{P_1\cap C_0, P_2\cap C_0,\cdots, P_{n}\cap C_0, C_1,\cdots,C_{k}\}$$  
	with $\operatorname{diam}(\mathcal{D})<\varepsilon$ and $\mu(\partial \mathcal{D})=0$.  
	
	By Proposition \ref{Prop_propertyofh} (1), we have
	$$h_{\mu}(f,\mathcal{D})\leq h_{\mu}(f,\mathcal{A})+H_{\mu}(\mathcal{D} | \mathcal{A}).$$  
    By equation (\ref{eq_conditional}),
	\begin{equation*}
		\begin{aligned}
			H_{\mu}(\mathcal{D} | \mathcal{A})  
			&=  \sum_{j=1}^k \mu(A_j) \sum_{i=1}^k  \phi\!\left(\frac{\mu(A_j \cap C_i)}{\mu(A_j)}\right)
			+  \sum_{j=1}^k \mu(A_j)\sum_{l=1}^n \phi\!\left(\frac{\mu(A_j \cap C_0\cap P_l)}{\mu(A_j)}\right).
		\end{aligned}
	\end{equation*}
     If $1 \leq i=j \leq k$, then  
	$$\frac{\mu(A_j\cap C_j)}{\mu(A_j)} = \frac{\mu(A_j)-\mu(A_j\backslash C_j)}{\mu(A_j)} \geq 1-\frac{\delta}{\mu(A_j)},$$  
	by (\ref{eq_xlnx2}), we have $\phi\!\left(\frac{\mu(A_j\cap C_j)}{\mu(A_j)}\right) <\eta.$  
	If $1 \leq i\neq j\leq k$, then  
	$$\frac{\mu(A_j\cap C_i)}{\mu(A_j)}\leq \frac{\delta}{\mu(A_j)},$$  
	and by (\ref{eq_xlnx1}), we have $\phi\!\left(\frac{\mu(A_j\cap C_i)}{\mu(A_j)}\right)<\eta.$  
	 If $i=0$ and $1 \leq j \leq k,1 \leq l \leq n$, since $\frac{\mu(A_j \cap C_0\cap P_l)}{\mu(A_j)}<\frac{k\delta}{\mu(A_j)}$,  
	by (\ref{eq_xlnx1}), we have $\phi\left(\frac{\mu(A_j \cap C_0\cap P_l)}{\mu(A_j)}\right) <\eta.$  
	Hence,
	$$h_{\mu}(f,\mathcal{D}) \leq h_{\mu}(f,\mathcal{A})+(k+n)\eta.$$
	Since $\eta$ is arbitrary, let $\eta \rightarrow 0$ we obtain  
	$$h_{\mu}(f,\mathcal{D})\leq  h_{\mu}(f,\mathcal{A})\leq  \inf_{\operatorname{diam}(\mathcal{A}) < \varepsilon} h_{\mu}(f,\mathcal{A}).$$  
	Thus,  
	$$\inf_{\operatorname{diam}(\xi) < \varepsilon,\ \mu(\partial \xi)=0 }h_{\mu}(f,\xi)\leq \inf_{\operatorname{diam}(\xi) < \varepsilon}h_{\mu}(f,\xi).$$

\end{proof}

\begin{Lem}\label{lemma-variationalequation}
Suppose $(X, f)$ is a dynamical system. Let $\left\{E_n\right\}$ be a sequence of $(n, \varepsilon)$-separated subsets for $X$ and define
	$$
	v_n : =\frac{1}{n\sharp E_n } \sum_{x \in E_n} \sum_{k=0}^{n-1} \delta_{f^k x},
	$$
	assume that $\lim_{n\rightarrow +\infty} v_n=\mu$, then for any $\varepsilon>0$, one has
	$$
	\limsup _{n \rightarrow \infty} \frac{1}{n} \ln \sharp E_n \leq \inf_{\operatorname{diam} (\xi) < \varepsilon }h_{\mu}(f,\xi).
	$$
\end{Lem}

\begin{proof}
	From the second part of the proof of \cite[Theorem 8.6]{water2001},  we obtain
	$$	\limsup _{n \rightarrow \infty} \frac{1}{n} \ln \sharp E_n \leq \inf_{\operatorname{diam}(\xi)< \varepsilon, \mu(\partial \xi)=0 }h_{\mu}(f,\xi). $$
	Combining with Lemma \ref{prop_infequal}, we complete the proof.
\end{proof}

\begin{Lem}\label{lemma-leq1}
    Suppose $(X, f)$ is a dynamical system  and $\mu\in \mathcal{M}_f(X)$. Then for any $\varepsilon>0$, one has 
	$$\inf _{F \ni \mu} \limsup_{n \rightarrow \infty} \frac{1}{n} \ln N(F, n, \varepsilon)\leq \inf_{\operatorname{diam} (\xi)<\varepsilon}h_{\mu}(f,\xi).$$
\end{Lem}
\begin{proof}
	 Suppose that there exist $\varepsilon, \delta>0$ such that 
	\begin{equation*}
	\inf _{F \ni \mu} \limsup_{n \rightarrow \infty} \frac{1}{n} \ln N(F, n, \varepsilon) \geq \inf_{\operatorname{diam} (\xi)<\varepsilon}h_{\mu}(f,\xi)+\delta.
	\end{equation*}
	Then there exists a decreasing sequence of convex closed neighborhoods $\left\{C_n\right\}_{n\in \mathbb{N}}$ satisfying 
	\begin{equation}\label{eq-lemmaN-1}
		     \bigcap_n C_n=\{\mu\}, \ 
			 \limsup_{n \rightarrow \infty} \frac{1}{n} \ln N\left(C_n, n, \varepsilon\right) \geq \inf_{\operatorname{diam} (\xi)<\varepsilon}h_{\mu}(f,\xi)+ \delta .
	\end{equation}
	Let $E_n \subset X_{n, C_n}$ be a $(n, \varepsilon)$-separated with maximal cardinality and define
	$$
	v_n =\frac{1}{n\sharp E_n} \sum_{x \in E_n} \sum_{k=0}^{n-1} \delta_{T^k x} \in C_n \text {.}
	$$
	By definition, we obtain $\lim _{n\rightarrow +\infty}\nu_n=\mu$. Combining with Lemma \ref{lemma-variationalequation},  it follows that
	$$
	\limsup _{n \rightarrow \infty} \frac{1}{n} \ln \sharp E_n=\limsup _{n \rightarrow \infty} \frac{1}{n} \ln N\left(C_n, n, \varepsilon\right) \leq \inf_{\operatorname{diam}(\xi)<\varepsilon} h_{\mu}(T, \xi),
	$$
	which contradicts (\ref{eq-lemmaN-1}).
\end{proof}

\begin{Lem}\label{lemma-leq2}
	Suppose $(X, f)$ is a dynamical system. 
	
	(1) Let $K \subset \mathcal{M}_{f}(X)$ be a closed subset, and 
	$G^K =\{x \in X: V_f(x) \cap K\neq \emptyset \}.$
	Then for any $\varepsilon>0$, 
	$$
	h^B_{\text {top}}\left(G^K,f,\varepsilon \right) \leq \sup_{\mu \in K} \inf _{\operatorname{diam}(\xi)<\varepsilon} h_\mu(f, \xi).
	$$
	
	(2) If $\mu \in \mathcal{M}_f(X)$, then for any $\varepsilon>0$,
	$$
	h^B_{\text {top}}\left(G_\mu,f,\varepsilon  \right) \leq \inf _{\operatorname{diam}( \xi)<\varepsilon} h_\mu(f, \xi).
	$$
	
	(3) Let $K  \subset \mathcal{M}_{f}(X)$ be a compact connected non-empty set. Then for any $\varepsilon>0$,
	$$
	h^B_{\text {top}}\left(G_K,f,\varepsilon \right) \leq \inf_{\mu \in K} \inf _{\operatorname{diam}(\xi)<\varepsilon} h_\mu(f,\xi).
	$$
\end{Lem}

\begin{proof}
	The item  (2) is a consequence of item (1),  since $G_{\mu}\subset G^{\mu}$. Obviously, $G_K \subset G^{\{\mu\}} $ for all $\mu \in K$,  then  item (3) can directly from  item (1). Thus, we only need to prove the statement item (1).
	
	Fix $\mu \in K$ and $s = \inf _{\operatorname{diam} (\xi)<\varepsilon} h_\mu(f, \xi).$
    Let $s^{\prime}=2 \delta+s>0$. Since $N(F, n, \varepsilon)$ is a non-increasing function of $\varepsilon$, by Lemma \ref{lemma-leq1},  for any  $\varepsilon>0$,  we have
	$$
	\inf _{F \ni \mu} \limsup _{n \rightarrow \infty} \frac{1}{n} \ln N(F, n, \varepsilon) \leq  \inf _{ \operatorname{diam}(\xi)<\varepsilon} h_\mu(f, \xi)  \text {. }
	$$
	In other words, 
	there exist a neighborhood $ F(\mu, \varepsilon)$ of $\mu$ and $M(F(\mu, \varepsilon)) \in \mathbb{N}$, so that
	$$
	\frac{1}{n} \ln N(F(\mu, \varepsilon), n, \varepsilon) \leq  \inf_{\operatorname{diam} (\xi)<\varepsilon} h_\mu(f, \xi)+\delta \text { for all } n \geq M(F(\mu, \varepsilon)).
	$$
	Then 
	$$N(F(\mu, \varepsilon), n, \varepsilon) \leq \exp\{n(\inf _{ \operatorname{diam} (\xi)<\varepsilon} h_\mu(f, \xi)+\delta) \} \text { for all } n \geq M(F(\mu, \varepsilon)).$$
	We know that maximal $(n, \varepsilon)$-separated subsets of a set $A$ are also $(n, \varepsilon)$-spanning subsets of $A$, for any $ n \geq M(F(\mu, \varepsilon))$,
	$$
	m \left(X_{n, F(\mu, \varepsilon)} , s^{\prime}, n, \varepsilon,f\right) \leq N(F(\mu, \varepsilon) , n, \varepsilon) \exp\{-s^{\prime} n\} \leq \exp\{-\delta n\}.
	$$
	Since $K$ is compact, given a fixed $\varepsilon>0$, 
	there exists a sequence subset $\{F(\mu_j,\varepsilon)\}_{j=1}^{m_{\varepsilon}}$ covering $K$ with $\mu_j \in K$. If $\left\{\mathcal{E}_n(x)\right\}$ has a limit-point in $K$, then $x\in A_M =\bigcup_{n \geq M} \bigcup_{j=1}^{m_{\varepsilon}} X_{n, F\left(\mu_j, \varepsilon\right)}$
	for any $M\in \mathbb{N}$. Thus, for $M \geq \max _{1 \leq j \leq m_{\varepsilon}} M\left(F\left(\mu_j, \varepsilon\right)\right)$,
	$$
	m\left(G^K  , s^{\prime}, M, \varepsilon,f \right)\leq m\left(A_M , s^{\prime}, M, \varepsilon, f\right) \leq m_{\varepsilon} \sum_{n \geq M} \exp\{-\delta n\},
	$$
	where $\sum_{n \geq M} \exp\{-\delta n\}$ is finite. It implies that
	$$
	h^B_{\text {top }}\left(G^K , f, \varepsilon\right) \leq s'=s+2\delta.
	$$
	Since $\delta$ is arbitrary, we get 
	$$h^B_{\text {top }}\left(G^K , f, \varepsilon\right) \leq \sup_{\mu \in K} \inf _{\operatorname{diam}(\xi)<\varepsilon} h_\mu(f, \xi).$$

\end{proof}

\begin{Lem}\label{lemma_G_mu}
	Let $f: X \rightarrow X$ be a continuous map on a compact metric space and $\mu \in \mathcal{ M}_f(X)$. If $Z \subset X$ and $\mu(Z)=1$, then there is a number $\varepsilon'$ such that
	for any $0<\varepsilon < \varepsilon'$, 
	$$ \inf_{ \operatorname{diam} (\xi)<\varepsilon}h_{\mu}(f,\xi)\leq h_{top}^B(Z, f, \varepsilon)+\varepsilon.$$
\end{Lem}

\begin{proof}
	For any $\varepsilon>0$,
	given a measurable Borel partition $\xi$ with  $\operatorname{diam}(\xi )<\varepsilon$,  there is an open cover $\mathcal{U}$ such that $H_\mu(\xi \mid \eta)<\varepsilon$ whenever $\eta$ is a finite Borel partition with $\eta  \prec \mathcal{U}$ \cite[Lemma 3]{Bowen1973}.  Let $M =\sharp \mathcal{U}$. 
	For each $n >0$,
	there exists a finite Borel partition $\alpha_n$ of $X$ such that $f^k \alpha_n \prec \mathcal{U}$ for all $k\in [0,n)$ and at most $nM$ sets in $\alpha_n$ can contain a point in all their closures \cite[Lemma 2]{Bowen1973}. Thus, $H_{\mu}(\xi| f^{k} \alpha_n)\leq \varepsilon$ for any $k \in [0, n )$. 	
	
	For each $x \in X$, let $I_m(x)=-\ln \mu(A)$, where $A \in  \bigvee_{i=0}^{m-1}   f^{-ni}\alpha_n$ contains $x$. By Shannon-McMillian-Breiman Theorem \cite{Parry1969}, there exists a $\mu$-integrable function $I(x)$ such that $I_m(x)/ m \rightarrow I(x)$ a.e., and $a_n =\int I(x) d \mu=h_\mu(f^n, \alpha_n)$. For any $\delta>0$, the set
	$$Z_\delta=\left\{y \in Z: I(y) \geq a_n-\delta\right\}$$
	has positive measure. 
	By Egorov's theorem, there is an $N\in \N$ so that
	$$ Z_{\delta, N}=\left\{y \in Z_\delta: I_m(y) / m \geq a_n-2 \delta, \ \forall m \geq N\right\}
	$$
	has positive measure.
	
	Let $\mathcal{D}_n$  be a finite open cover of $X$, such that each member of $\mathcal{D}_n$ intersects at most $nM$ member of  $\alpha_n$,  and let $\varepsilon_0 : =\operatorname{Leb}(\mathcal{D}_n)$.
	For any $\varepsilon< \varepsilon_0$,  
	let  $\mathcal{ E}= \{B_{n_i}( x_i,\varepsilon,f^n)\}_{i\in I}$ be an open cover of $Z$ satisfying
	 $\{B_{n_i}(x_i,\varepsilon,f^n)\}_{i\in I} \prec  \bigvee_{j=0}^{n_i-1}f^{- jn} \mathcal{D}_n $ and $\min\{n_i : i \in I\}\geq  N$.
	If $\beta \in \bigvee_{i=0}^{n_i-1} f^{-in} \alpha_n$ such that  $\beta\cap Z_{\delta, N}\neq \emptyset$, then 
	$$\mu(\beta) \leq \exp \left((-a_n+2 \delta ) n_i \right).$$  
	$B_{n_i}(x_i,\varepsilon,f^n) \cap Z_{\delta, N}$ is covered by at most ${(nM)}^{n_i}$ such $\beta$, then
	$$\mu(B_{n_i}(x_i,\varepsilon,f^n)\cap Z_{\delta,N}) \leq \exp \left((\ln nM-a_n+2 \delta) n_i \right).$$
	For $\lambda=-\ln nM+a_n-2 \delta$, we have
	$$
	\sum_{i\in I} \exp\{-\lambda n_i\} \geq \sum_{i\in I} \mu(B_{n_i}(x_i,\varepsilon,f^n)\cap Z_{\delta,N}) \geq \mu(Z_{\delta,N}).
	$$
	Letting $\mathcal{E}$ vary, it holds that $m(Z,\lambda,N,\varepsilon,f^n)\geq \mu(Z_{\delta,N})>0$, then
	$h_{top}^B(Z,f^n,\varepsilon) \geq  \lambda.$ 
	Thus, letting $\delta \rightarrow 0$, we obtain
	$$h_{top}^B(Z,f^n,\varepsilon) \geq h_{\mu}(f^n,\alpha_n)-\ln n M.$$
	By Proposition \ref{Prop_propertyofdim} (2) and Proposition \ref{Prop_propertyofh} (3),  we get
	$$
	\begin{aligned} 
		h_{\mu}(f,\xi) &=\frac{1}{n}h_{\mu}(f^n,\bigvee_{k=0}^{n-1}f^{-k}\xi)\leq \frac{1}{n} h_\mu\left(f^n, \alpha_n\right)+\frac{1}{n} H_\mu\left(\bigvee_{k=0}^{n-1}f^{-k}\xi \mid \alpha_n \right) \\
		& \leq n^{-1}\left( h_{top}^B\left(Z, f^n, \varepsilon \right)+\ln (n M)\right)+n^{-1} \sum_{k=0}^{n-1} H_\mu\left(f^{-k} \xi \mid \alpha_n\right) \\ & \leq h_{top}^B(Z, f, \varepsilon)+\frac{1}{n}\ln (n M)+n^{-1} \sum_{k=0}^{n-1} H_\mu\left(\xi \mid f^k \alpha_n\right) \\ & \leq h_{top}^B(Z, f,  \varepsilon)+\frac{1}{n}\ln (n M)+\varepsilon.
	\end{aligned}
	$$
	Let $n \rightarrow +\infty$,  we obtain the conclusion that
	$$\inf_{ \operatorname{diam} (\xi)<\varepsilon}h_{\mu}(f,\xi)\leq h_{top}^B(f, Z, \varepsilon)+\varepsilon.$$
	
\end{proof}

\subsection{Proof of Theorem \ref{Maintheorem_G_mu}.}
By the Birkhoff Ergodic Theorem, for an ergodic measure $\mu$, one has $\mu(G_{\mu})=1$. According to Lemma \ref{lemma_G_mu},  there exists $\varepsilon'>0$, 
such that for any $0 < \varepsilon <\varepsilon'$, 
$$
\frac{1}{|\ln\varepsilon|}h_{top}^B(G_{\mu},f,\varepsilon)
\geq  \frac{1}{|\ln\varepsilon|} \inf_{\operatorname{diam}(\xi) <\varepsilon}h_{\mu}(f,\xi)- \frac{\varepsilon}{|\ln \varepsilon|}.
$$
Then let $\varepsilon\rightarrow 0$, we get
$$ \overline{\operatorname{mdim}}^B_M \left(G_{\mu},f,d\right) 
\geq  \limsup_{\varepsilon \rightarrow 0} \frac{1}{|\ln\varepsilon|} \inf_{\operatorname{diam}(\xi) <\varepsilon}h_{\mu}(f,\xi).$$ Now we prove  the reverse inequality. 
For any $\gamma>0$,  there exists $\varepsilon>0$ such that 
\begin{equation*}
\overline{\operatorname{mdim}}^B_M \left(G_{\mu},f,d\right) \leq \frac{1}{|\ln\varepsilon|}h_{top}^B(G_{\mu},f,\varepsilon)+\gamma.
\end{equation*}
By Lemma \ref{lemma-leq2} (2),  we get 
$$
\overline{\operatorname{mdim}}^B_M \left(G_{\mu},f,d\right) \leq
\frac{1}{|\ln \varepsilon|}h^B_{\text {top}}\left(G_\mu,f,\varepsilon  \right)+\gamma \leq  \frac{1}{|\ln \varepsilon|} \inf _{\operatorname{diam} (\xi)<\varepsilon} h_\mu(f, \xi)+\gamma.
$$
Thus,
$$ \overline{\operatorname{mdim}}^B_M \left(G_{\mu},f,d\right) 
\leq  \limsup_{\varepsilon \rightarrow 0} \frac{1}{|\ln\varepsilon|} \inf_{\operatorname{diam}(\xi) <\varepsilon}h_{\mu}(f,\xi).$$
	Since the proof of the inequality for
	 $\underline{\operatorname{mdim}}^B_M \left(G_{\mu},f,d\right)$ is similar,  we omit it here.

\subsection{Proof of Theorem \ref{MainTheorem}}\label{Section_4}

\subsubsection{ Upper bound for	$\overline{\operatorname{mdim}}^B_{M}\left(G^C_K, f, d\right)$ and $\underline{\operatorname{mdim}}^B_{M}\left(G^C_K, f, d\right)$ }

	By Lemma \ref{lemma-leq2} (3), it holds that $$h_{top}^B(G_K,f,\varepsilon)\leq \inf_{\mu \in K} \inf _{\operatorname{diam}(\xi)<\varepsilon} h_\mu(f, \xi) \textit{ for any } \varepsilon>0.$$
	Since $G^C_K\subset G_K$, we get 
	$$
	\begin{aligned}
		\overline{\operatorname{mdim}}^B_{M}\left(G^C_K, f, d\right)&=\limsup_{\varepsilon\rightarrow 0}\frac{h_{top}^B(G^C_K,f,\varepsilon)}{|\ln \varepsilon|}\leq \limsup_{\varepsilon\rightarrow 0}\frac{h_{top}^B(G_K,f,\varepsilon)}{|\ln \varepsilon|}\\
	&\leq \limsup _{\varepsilon \rightarrow 0}  \frac{1}{|\ln \varepsilon|} \inf_{\mu \in K} \inf _{\operatorname{diam}(\xi)<\varepsilon} h_\mu(f, \xi).
	\end{aligned}
	$$
Similarly, $$\underline{\operatorname{mdim}}^B_{M}\left(G^C_K, f, d\right) \leq  \liminf_{\varepsilon \rightarrow 0}  \frac{1}{|\ln \varepsilon|} \inf_{\mu \in K} \inf _{\operatorname{diam}(\xi)<\varepsilon} h_\mu(f, \xi).$$

\subsubsection{ Lower bound for	$\overline{\operatorname{mdim}}^B_{M}\left(G^C_K, f, d\right)$ and $\underline{\operatorname{mdim}}^B_{M}\left(G^C_K, f, d\right)$ }

\begin{Lem}\label{Lemma_ergodicdecomposition} \cite[Lemma 6.2]{PS2007}
Let $\varepsilon>0, \delta>0$. Let $v \in \mathcal{M}_f(X)$ and $v=\int \tau d \hat{\mu}(\tau)$ be its ergodic decomposition. Then, for any $\Delta>0$ we can find a finite convex combination $\sum_{i=1}^p a_i \mu_i$ of ergodic measures such that
$$
d\left(\nu, \sum_{i=1}^p a_i \mu_i\right) \leq \Delta,
$$
and
$$
\int \overline{PS}\left(\tau, \delta, \varepsilon\right) d \hat{\mu}(\tau) \leq \sum_{i=1}^p a_i \overline{PS}\left(\mu_i, \delta, \varepsilon\right) .
$$
The $\left\{a_i\right\}$ can be chosen to be rational numbers. 
\end{Lem}

\begin{Lem}\cite[Page 944]{PS2007} \label{lemma-KK}
	For any compact connected non-empty set $K \subseteq \mathcal{M}_f(X)$,
	there exists a sequence $\left\{\alpha_1, \alpha_2, \ldots \right\}$ in $K$ such that
	$
	\overline{\left\{\alpha_j: j \in \mathbb{N}^{+}, j>n\right\}}=K, \forall n \in \mathbb{N}^{+} \text {, and } \lim _{j \rightarrow \infty} \rho \left(\alpha_j, \alpha_{j+1}\right)=0.
	$
\end{Lem}

\begin{Lem}\cite[Lemma 2.1]{PS2007}\label{lemma-estimate}
	Suppose $(X, f)$ is a dynamical system with $g$-almost product property. Given $x_1, \ldots, x_k \in X,\ \varepsilon_1, \ldots, \varepsilon_k$ and $n_1 \geq m\left(\varepsilon_1\right), \ldots, n_k \geq m\left(\varepsilon_k\right)$. Assume that there are $\nu_j \in \mathcal{M}_f(X)$ and $\zeta_j>0$ satisfying
	$$
	\rho(\mathcal{E}_{n_j}\left(x_j\right),  \nu_j ) \leq \zeta_j,\  j=1,2, \ldots, k.
	$$
	Then for any $z \in \cap_{j=1}^k f^{-Q_{j-1}} B_{n_j}\left(g ; x_j, \varepsilon_j\right)$ and any probability measure $\alpha \in \mathcal{M}(X)$,
	$$
	\rho \left(\mathcal{E}_{Q_k}(z), \alpha\right) \leq \sum_{j=1}^k \frac{n_j}{Q_k}\left(\zeta_j+\varepsilon_j+\frac{g\left(n_j\right)}{n_j}+\rho \left(\nu_j, \alpha\right)\right),
	$$
	where $Q_0=0,\ Q_i=n_1+\cdots+n_i$.
\end{Lem}

\begin{Prop}\label{prop-G^C-lower}
	Under the hypotheses of Theorem \ref{MainTheorem},  we have
	$$
	\overline{\operatorname{mdim}}^B_{M}\left(G^C_K, f, d\right)\geq \limsup _{\varepsilon \rightarrow 0} \frac{1}{|\ln \varepsilon|} \inf_{\mu \in K} \inf _{\operatorname{diam}(\xi)<\varepsilon} h_\mu(f, \xi),
	$$
	and
	$$\underline{\operatorname{mdim}}^B_{M}\left(G^C_K, f, d\right)\geq\liminf_{\varepsilon \rightarrow 0} \frac{1}{|\ln \varepsilon|} \inf_{\mu \in K} \inf _{\operatorname{diam}(\xi)<\varepsilon} h_\mu(f, \xi).$$
\end{Prop}

\begin{proof}
	Since $K\subseteq  \operatorname{conv}\{\mu_1,\ldots,\mu_m\}$ is connected and compact,	by Lemma \ref{lemma-KK}, there exists $$\{\alpha_1,\ldots,\alpha_n,\ldots\} \subseteq K$$ such that 
	$$\overline{\left\{\alpha_j: j \in \mathbb{N}^{+}, j>n\right\}}=K,\ \forall n \in \mathbb{N}^{+} \text {, and } \lim _{j \rightarrow \infty} \rho \left(\alpha_j, \alpha_{j+1}\right)=0.$$
 Fix $i\in \mathbb{N}$, let $\mu_i = \int_{\mathcal{ M}_f^e(X)} \tau d \hat{\mu}_i(\tau)$ be its ergodic decomposition.
 By Theorem \ref{Thm_entropyformula_PS},  if  $\nu\in \mathcal{M}^e_f(X)$,
	$$
	\lim_{\varepsilon \rightarrow 0} \lim _{\delta \rightarrow 0} \overline{PS}(\nu, \delta, \varepsilon)=h_{\nu}(f).
	$$
	Note that $\overline{PS}\left(\tau, \delta, \varepsilon\right)$ is non-increasing in $\delta$ and $\varepsilon$. By the Monotone-convergence theorem and the affine character of the entropy,  we  have
	$$
	\lim _{\varepsilon \rightarrow 0} \lim _{\delta \rightarrow 0} \int \overline{PS}\left(\tau, \delta, \varepsilon\right) d \hat{\mu}_i(\tau) =	\int \lim _{\varepsilon \rightarrow 0} \lim _{\delta \rightarrow 0} \overline{PS} \left(\tau, \delta, \varepsilon \right) d \hat{\mu}_i(\tau)
	=\int h_{\tau}\left(f\right) d \hat{\mu}_i(\tau)=h_{\mu_i}(f) .
	$$
	Denote $$S=\limsup _{\varepsilon \rightarrow 0} \frac{1}{|\ln \varepsilon|} \inf_{\mu \in K} \inf _{\operatorname{diam}(\xi)<\varepsilon} h_\mu(f, \xi).$$ For any $\gamma>0$, there exists a sufficiently small $\varepsilon^*>0$ such that
	\begin{equation}\label{eq-lowebound-1}
S-3\gamma  {\leq} \frac{\inf_{\mu \in K} \inf _{\operatorname{diam}(\xi)<4\varepsilon^*} h_\mu(f, \xi) -3\gamma }{|\ln 4 \varepsilon^*|},
	\end{equation}
	\begin{equation}\label{eq-lowerbound-2}
 \frac{h_{top}^B(G^C_K, f,\frac{\varepsilon^*}{8})}{|\ln 4 \varepsilon^*|}\leq  \frac{h_{top}^B(G^C_K, f,\frac{\varepsilon^*}{8})}{|\ln\frac{\varepsilon^*}{8}|} {\leq} 	
\overline{\operatorname{mdim}}^B_{M}\left(G^C_K, f, d\right)+\gamma.
	\end{equation}
	and 
		\begin{equation}\label{eq-lowerbound-3}
	\lim_{\delta\rightarrow0} \int \overline{PS}\left(\tau, \delta, \varepsilon^* \right) d \hat{\mu}_i(\tau) \geq h_{\mu_i} (f)-\gamma/2 \textit{ for each } 1 \leq i \leq m.
	\end{equation}
	\begin{Lem}\label{lemma-Lambda}
		There exists $\delta^*>0$, such that
		 for each $\mu_i(1 \leq i \leq m)$ and any neighborhood $F_i \subset \mathcal{M}(X)$ of $\mu_i$,  there exists $N_{F_i, \delta^*,\varepsilon^*,\gamma}\in \mathbb{N}$, such that for any $n\geq N_{F_i, \delta^*,\varepsilon^*,\gamma}$, there exists a $(\delta^*/2,n, \varepsilon^*/2)$-separated  set $\Lambda^i_n \subset X_{n,F_i}$ satisfying
		$$ \sharp \Lambda^i_n \geq \exp\{n ({\inf_{\operatorname{diam}(\xi) <4\varepsilon^*}h_{\mu}(f,\xi)-\gamma/2})\}.$$
	\end{Lem}
	
	\begin{proof}
		By (\ref{eq-lowerbound-3}),
	    there exists sufficient small $\delta_i>0$,   such that  
	    \begin{equation}\label{eq_PS1}
	    \int \overline{PS}\left(\tau, \delta_i, \varepsilon^* \right) d \hat{\mu}_i(\tau) \geq h_{\mu_i} (f)-\gamma/2. 
	    \end{equation}
	    	Choose $\ka>0$ such that $B(\mu_i,\ka)\subset F_i$.
	    	By Lemma \ref{Lemma_ergodicdecomposition}, 
	    there exists a finite convex combination of ergodic measures with rational coefficients
		$\nu_i=\sum_{j=1}^{p} b^i_{j} \nu^i_{j}$
		with $\sum_{j=1}^{p} b^i_{j}=1$, so that  $$\rho(\nu_i,\mu_i)\leq \ka /4$$
		 and
		\begin{equation}
			\int \overline{PS}\left(\tau, \delta_i, \varepsilon^* \right) d \hat{\mu}_i(\tau)
		 \leq \sum_{j=1}^{p} b^i_{j} \overline{PS}\left(\nu^i_j, \delta_i, \varepsilon^*\right).
		\end{equation}
		Obviously, $\inf_{diam(\xi)<4\varepsilon^*} h_{\mu_i}(f,\xi) \leq h_{\mu_i}(f)$,  combining with (\ref{eq_PS1}), it follows that
	\begin{equation}\label{lemma-separated1}
		 \inf_{diam(\xi)<4\varepsilon^*} h_{\mu_i}(f,\xi)-\gamma/2
		\leq \sum_{j=1}^{p} b^i_{j} \overline{PS}\left(\nu^i_j, \delta_i, \varepsilon^*\right).
		\end{equation}

		Let $ \varepsilon^i_1,\cdots, \varepsilon^i_p$ be  positive numbers such that  for any $1\leq  j \leq p$, 
		%$\lim _{k \rightarrow \infty} \varepsilon_k=0$
		$\varepsilon^i_j<\operatorname{mim}\left\{\frac{\varepsilon^*}{4}, \frac{\ka}{4}\right\}$.  	Let $m: \mathbb{R}^{+} \rightarrow \mathbb{N}$ be the non-increasing function by the $g$-almost product property. 
			Fix $1\leq  j \leq p$. By the definition of $\overline{PS}\left(\nu^i_j,  \delta_i, \varepsilon^*\right),$
	 for any neighborhood $B(\nu^i_j,\ka/4)$ of $\nu^i_j$,
		there exists $N_{\nu^i_j,\ka, \delta_i,\varepsilon^*}\in \N$,  choose  $n \in \N$
		satisfying  $b^i_{j} n \in \mathbb{N}$, $b_{j}^i n>\max \{ \max_{1 \leq j \leq p}\{ N_{\nu^i_j,\ka, \delta_i,\varepsilon^*} \}, m(\varepsilon_j^i)\}$, and
		\begin{equation}\label{eq-lemmaseparated1}
		\frac{g\left(b^i_jn\right)}{b^i_j n} \leq \operatorname{mim}\left\{\frac{ \delta_i}{4}, \frac{\ka}{4}\right\}, 
		\end{equation}
		such that
		$$
		N(B(\nu^i_j, \ka/4 ), \delta_i, b^i_{j} n, \varepsilon^*) \geq \exp\{b^i_{j} n\left(\overline{PS}\left(\nu^i_j, \delta_i, \varepsilon^* \right)-\gamma/2\right)\} .
		$$
		Denote $\Gamma_i  =\prod_{j=1}^{p} \Gamma^i_{j}$, where  $\Gamma^i_{j}= \Gamma(\delta_i,b^i_j n,\varepsilon^*)$ is  a $\left(\delta_i,b^i_j n,  \varepsilon^*  \right)$-separated set of $X_{b^i_j n,B(\nu^i_j,\ka/4)}$ with the largest cardinality.  Then
		\begin{eqnarray}\label{eq-lemmaseparated2}
		\begin{aligned}
		\sharp \Gamma_i=\prod_{j=1}^{p}\sharp \Gamma_{j} ^i &\geq \exp\{ \sum_{j=1}^{p} b^i_j n\left(\overline{PS}\left(\nu^i_j, \delta_i, \varepsilon^*\right)-\gamma/2\right)\}\\
		&\overset{(\ref{lemma-separated1})}{\geq}  \exp\{n(	\inf_{\operatorname{diam} (\xi) <4 \varepsilon^*}h_{\mu}(f,\xi)-\gamma/2-\gamma/2) \}\\
		&\geq \exp\{n (\inf_{\operatorname{diam} (\xi) <4 \varepsilon^* }h_{\mu}(f,\xi)-\gamma)\}.
		\end{aligned}
		\end{eqnarray}
		The elements of $\Gamma_i$ is $\bar{x}^i=\left(x^i_{1}, \ldots, x^i_{p}\right)$ with $x^i_j \in \Gamma^i_j$ such that  $ \mathcal{E}_{b^i_j n}\left(x^i_j\right) \in B\left(\nu^i_j, \ka/4\right),$
		and the set
		$$
		\Lambda^i_n : =\bigcap_{j=1}^{p} \bigcup_{x^i_j\in \Gamma^i_j} f^{-\left(b^i_{1}+\cdots+b^i_{j-1}\right) n} B_{b^i_j n}\left(g; x^i_j, \varepsilon^i_j\right)
		$$
		 with $b^i_{0} : =0$ is a non-empty set by $g$-almost product property.
		Claim that the map $\sigma:\Gamma_i \rightarrow  \Lambda^i_n  $ is bijective.
	Indeed,  for any $\bar{x}^i, \bar{y}^i \in \Gamma$ with $x^i_k\neq y^i_k$, 
		since $x^i_k$ and $y^i_k$ are $\left(\delta_i, b^i_k n, \varepsilon^*\right)$-separated, that is,
		$$
		\sharp \left\{0 \leq l<b^i_k n: d\left(f^l x^i_k, f^l y^i_k \right)>\varepsilon^* \right\} \geq \delta_i b^i_k n.
		$$
		$\sigma (\bar{x}^i)$ traces $x^i_k$ and  $\sigma\left(\bar{y}^i \right)$ traces $y^i_k$ both on $[\sum_{j=1}^{k-1} b^i_j n, \sum_{j=1}^{k} b^i_j n-1]$,
	  namely,
		$$
		\sharp \{0 \leq l<b^i_k n : d(f^{l} y^i_{k}, f^{\sum_{j=1}^{k-1} b^i_j  n+l} \sigma(\bar{y}^i)) \leq \varepsilon^i_k\} \geq b^i_k n-g(b^i_k n),
		$$
		$$
		\sharp \{0 \leq l<b^i_k n : d(f^{l} x^i_k, f^{\sum_{j=1}^{k-1} b^i_j n+l} \sigma(\bar{x}^i)) \leq \varepsilon^i_k\} \geq b^i_k n-g(b^i_k n).
		$$
		Then
		$$\sharp \{ \sum_{j=1}^{k-1}b_j^i n \leq l \leq \sum_{j=1}^{k} b_j^i n-1:  d(f^l \sigma(\bar{x}^i),f^l \sigma(\bar{y}^i))\geq \frac{\varepsilon^*}{2}\} \geq \delta_i b_k^i n-2g(b_k^i n)\overset{(\ref{eq-lemmaseparated1})}{\geq} \frac{\delta_i}{2}b_k^in.$$
		Let $\delta^*=\min_{1 \leq k \leq p}\{\delta_i \}$, 	thus,  $\Lambda^i_n$  is $(\frac{\delta^*}{2},n, \frac{\varepsilon^*}{2})$-separated set.
		Then   $$\sharp \Lambda^i_n=\sharp \Gamma_i \overset{(\ref{eq-lemmaseparated2})}{\geq} \exp\{n (\inf_{\operatorname{diam}(\xi)<4\varepsilon^*}h_{\mu}(f,\xi)-\gamma)\}.$$
		What's more, for any $z \in \Lambda^i_n$, it derives from Lemma \ref{lemma-estimate} that
		\begin{eqnarray*}
			\rho(\mathcal{E}_{n}(z),\mu) 
			&\leq& \rho(\sum_{j=1}^{p}b^i_{j}\mathcal{E}_{b_{j}n}(z),\sum_{j=1}^{p}b^i_{j}\mathcal{E}_{b^i_{j}n}(f^{\sum_{l=1}^{j-1}b^i_{l}n}x^i_{j}))+\\
			&&\rho(\sum_{j=1}^{p}b^i_{j}\mathcal{E}_{b^i_{j}n}(f^{\sum_{l=1}^{j-1}b^i_{l}n}x^i_{j}),\sum_{j=1}^{p}b^i_{j}\nu^i_{j})+ \rho(\sum_{j=1}^{p}b^i_{j}\nu^i_{j},\mu)\\
			&\leq& \sum_{j=1}^{p}b^i_{j}(\varepsilon^i_j+ \frac{g(b^i_{j}n)}{b^i_{j}n}+\ka/4)+\ka/4 \\
			&\leq& \ka.
		\end{eqnarray*}
		Thus, $ \Lambda_n^i
		\subseteq X_{n,B(\mu_i,\ka)}\subseteq X_{n,F_i}$.
	\end{proof}

	 Let $ H^*=\inf_{\mu \in K} \inf _{\operatorname{diam}(\xi)<4\varepsilon^*} h_\mu(f, \xi) -\gamma$, and $\{\xi_k\}$, $\{\beta_k\}$, and $\{\varepsilon_k\}$ be  strictly decreasing sequences such that $\lim_{k\rightarrow+\infty} \xi_k = 0$ satisfying
	\begin{equation}\label{eq-xi}
	\xi_1<\min\{\frac{\varepsilon^*}{2},\frac{\gamma}{H^*}\}, 
	\end{equation}
   $\lim_{k\rightarrow+\infty} \beta_k = 0$ satisfying
	\begin{equation}\label{eq-beta}
	\beta_1\leq \frac{\varepsilon^*}{16},
	\end{equation}
   $\lim_{k\rightarrow +\infty}\varepsilon_k=0$ satisfying
		\begin{equation}\label{eq-epsilon}
	\varepsilon_k<\min\{\frac{\varepsilon^*}{8},\frac{\xi_k}{8}\} \textit{ for any } k\in \mathbb{N}^+.
	\end{equation}
	%and $\lim_{k\rightarrow 0}\gamma_k =0$ with $\gamma_1\leq \gamma.$
	%Note that $\operatorname{conv}\{\mu_1,\cdots,\mu_m\}$ is a closed set, 
	Fix $k\in \mathbb{N}^+$.  
	There exists a subset $\{c_i^k\}_{i=1}^m\subseteq [0,1]$ such that for any partition $\xi$ with $ \operatorname{diam}( \xi )< 4 \varepsilon^*$, one has $\alpha_k=\sum_{i=1}^m c_i^k \mu_i$ and $h_{\alpha_k}(f,\xi)=\sum_{i=1}^m c_i^k h_{\mu_i}(f,\xi)$. By the denseness of the rational numbers, we can choose each $c_i^k=\frac{b_i^k}{b^k}$ with $b_i^k\in \N$ and $\sum_{i=1}^k b_i^k=b^k$,  such that 
\begin{equation}\label{eq_closeconvex}
	\rho(\alpha_k, \sum_{i=1}^m  \frac{b_i^k}{b^k} \mu_i)\leq \frac{\xi_k}{4} \textit{ and }   h_{\alpha_k}(f,\xi)\leq  \sum_{i=1}^m  \frac{b_i^k}{b^k} h_{\mu_i}(f,\xi)+\frac{\gamma}{2}.
\end{equation}
	Let $m: \mathbb{R}^{+} \rightarrow \mathbb{N}$ be the non-increasing function by the $g$-almost product property. By Proposition \ref{lemma-APdense} the almost periodic set $AP(X)$ is dense in $C_f(X)$, then there is a finite set $\Theta_k=\left\{x_1^k, x_2^k, \cdots, x_{t_k}^k\right\} \subseteq A P(X)$ and $L_k \in \mathbb{N}$ such that $\Theta_k$ is $\beta_k$-dense in $C_f(X)$,  and for any $1 \leq i \leq t_k$, any $l \geq 1$, there is $n \in\left[l, l+L_k\right]$ such that $f^n\left(x_i^k\right) \in B\left(x_i^k, \beta_k\right)$. This implies that for any $1 \leq i \leq t_k$,
	\begin{equation}\label{eq-AP1}
	\frac{\sharp \left\{0 \leq n \leq l L_k: d\left(f^n x_i^k, x_i^k\right) \leq \beta_k \right\}}{l L_k} \geq \frac{1}{L_k} .
	\end{equation}
	Take $l_k$ large enough such that
	\begin{equation}\label{eq-AP2}
	l_k L_k \geq m\left(\beta_k\right),\  \frac{g\left(l_k L_k\right)}{l_k L_k}<\frac{1}{4 L_k} .
	\end{equation}
	We may assume that the sequences of $\left\{t_k\right\},\left\{l_k\right\},\left\{L_k\right\}$ are strictly increasing.
	By Lemma \ref{lemma-Lambda},  there exists $\delta^*>0$, such that  for any $\mu_i \in \mathcal{ M}_f(X)$, and the  neighborhood $B(\mu_i,\frac{\xi_k}{4}) \subseteq \mathcal{M}(X)$, there exists  $N_{B(\mu_i,\frac{\xi_k}{4}),\delta^*,\varepsilon^*, \gamma}$, such that for any large enough $n^k\in \N$ satisfying
	\begin{equation}\label{eq-lowerbound-n0}
		b_i ^kn^k \geq m\left(\varepsilon_i \right),
	b_i^kn^k>\max_{1\leq i\leq m }\{N_{B(\mu_i,\frac{\xi_k}{4}),\delta_i,\varepsilon^*,\gamma} \},\ \frac{g(b_i^k n^k)}{b_i^k n^k}\leq \min\{\frac{\delta^*}{8},\frac{\xi_k}{8}\} \textit{ for any } 1 \leq i \leq m,
	\end{equation}
	\begin{equation}\label{eq-lowerbound-n1}
	b^kn^k \geq   m(\beta_k),
	\frac{g(b^kn^k)}{b^kn^k}\leq \min\{\beta_k,\frac{\delta^*}{8}\},
	\end{equation}
	\begin{equation}\label{eq-lowerbound-n2}
	\frac{\delta^*b^kn^k}{4} >2g(b^kn^k)+1,
	\end{equation}
	\begin{equation}\label{eq-lowerbound-n3}
	\frac{t_k l_k L_k}{b^kn^k} \leq \xi_k,
	\end{equation}
	and
	\begin{equation}\label{eq-lowerbound-n4}
	H^*  b^kn^k \geq (H^*-\gamma)\left(b^kn^k+t_k l_k L_k\right).
	\end{equation}
	There is a $(\delta^*/2,b_i^kn^k,\varepsilon^*/2)$-separated set $\Lambda_{b_i^kn^k}^i \subset X_{b_i^kn^k,B(\mu_i,\frac{\xi_k}{4})}$ with
	\begin{equation}\label{eq-lowerbound-22}
		\sharp \Lambda_{b_i^kn^k}^i \geq \exp\{b_i^kn^k(\inf_{ \operatorname{diam} (\xi)<4\varepsilon^*}h_{\mu_i}(f,\xi)-\gamma/2)\}.
	\end{equation}

	Define $\Lambda_k=\prod_{i=1}^{m} \Lambda_{b_i^kn^k}^i $,
	the elements of $\Lambda_k$ is $
	\hat{y}_k=(y_{1}^k, \ldots, y_{m}^k ) \text { with } \mathcal{E}_{b_i^kn^k}(y_{i}^k) \in B(\mu_{i}, \frac{\xi_k}{4}),$
    the set
	$$
	\Delta_{b^k n^k} : =\bigcap_{j=1}^{m} \bigcup_{y_j^k\in \Lambda^j_{b_j^kn^k}} f^{-M_{j-1}^k}B_{b_{j}^k n^k}\left(g; y_{j}^k, \varepsilon_j \right)\  \text { with }M_{j}^k : =\sum_{l=1}^{j}b_l^k n^k \textit{ and } M_{0}^k : =0
	$$
	is a non-empty closed set by $g$-almost product property.
We claim that  the map $\phi_k: \Lambda_k \rightarrow 	\Delta_{b^kn^k}$ is bijective. Indeed,  if  $  x^k_l \neq y^k_l \in \Lambda^l_{b_l^kn^k}$ ($0\leq l \leq m-1$), one has
	$$\sharp \{ \sum_{i=1}^{l-1}b_{i}^kn^k \leq j\leq \sum_{i=1}^{l}b_{i}^k n^k-1:  \rho(f^j \phi_k(\hat{x}_k),f^j \phi_k(\hat{y}_k))\geq \frac{\varepsilon^*}{4}\} \geq\frac{ \delta^* b_{l}^k n^k}{2}-2g(b_{l}^k n^k)\overset{(\ref{eq-lowerbound-n0})}{\geq} \frac{\delta^*}{4}b_{l}^kn^k.$$
	Thus, 
	$\Delta_{b^k n^k}$ is a $(\frac{\delta^*}{4},b^kn^k, \frac{\varepsilon^*}{4})$-separated set. 
    Then 
	\begin{equation}\label{eq-lowerbopund-4}
	\begin{aligned}
	\sharp \Delta_{b^kn^k}&=\sharp \Lambda_k \overset{(\ref{eq-lowerbound-22})}{\geq} \exp\{b^kn^k\sum_{i=1}^{m}\frac{b_i^k}{b^k} ( \inf_{ \operatorname{diam} (\xi)<4\varepsilon^*}h_{\mu_i}(f, \xi) -\gamma/2)\}\\
&\overset{(\ref{eq_closeconvex})}{\geq} \exp\{b^k n^k( \inf_{ \operatorname{diam} (\xi)<4\varepsilon^*}h_{\alpha_k}(f, \xi)-\gamma/2-\gamma/2)\}	\\
&\geq \exp\{b^kn^k( \inf_{ \operatorname{diam} (\xi)<4\varepsilon^*}h_{\alpha_k}(f, \xi)-\gamma)\}\\
& \geq  \exp\{b^kn^k H^*\}.
\end{aligned}
	\end{equation}
	What's more, for any $z \in  \Delta_{b^kn^k}$, by lemma \ref{lemma-estimate} and  (\ref{eq_closeconvex}),  one can see that
	\begin{eqnarray*}
		\rho(\mathcal{E}_{b^kn^k}(z),\alpha_k)
		&\leq& \rho(\sum_{i=1}^{m}\frac{b_{i}^k}{b^k}\mathcal{E}_{b_{i}^k n^k}(f^{M_{i-1}^k}z),
		\sum_{i=1}^{m}\frac{b_{i}^k}{b^k}\mathcal{E}_{b_{i}^k n^k}(y_{i}^k))+
		\rho( \sum_{i=1}^{m}\frac{b_{i}^k}{b^k}\mathcal{E}_{b_{i}^k n^k}(y_{i}^k), \sum_{i=1}^{m}\frac{b_{i}^k}{b^k}\mu_{i})\\
	   & & +\rho(\sum_{i=1}^{m}\frac{b_{i}^k}{b^k}\mu_{i},\alpha_k)\\
	         	&\leq& 
		\sum_{i=1}^{m}\frac{b_{i}^k}{b^k}(\varepsilon_i+ \frac{g(b_i^kn^k)}{b_i^kn^k} )+\frac{\xi_k}{4}+\frac{\xi_k}{4}\\
		 &\leq&  \xi_k.
	\end{eqnarray*}
	So,  we have $\Delta_{b^k n^k}
	\subseteq X_{b^kn^k,B(\alpha_k,\xi_k)}$. 
	Thus, let $M_k : =b^kn^k$ and $\Delta_{k} : =\Delta_{b^k n^k}$. we find a  $(\frac{\delta^*}{4},M_k, \frac{\varepsilon^*}{4})$-separated set $\Delta_{k} \subseteq  X_{M_k,B(\alpha_k,\xi_k)}$  and 
	\begin{equation}\label{eq-Delta_MK}
	\sharp \Delta_{k} \geq \exp\{M_k H^*\} .
	\end{equation}

	We choose a strictly increasing $\left\{N_k\right\}$, with $N_k \in \mathbb{N}$, so that
	\begin{equation}\label{eq-N_k-1}
	M_{k+1}+t_{k+1} l_{k+1} L_{k+1} \leq \xi_k \sum_{j=1}^k\left(M_j N_j+t_j l_j L_j\right),
	\end{equation}
	and
	\begin{equation}\label{eq-N_k-2}
	\sum_{j=1}^{k-1}\left(M_j N_j+t_j l_j L_j\right) \leq \xi_k \sum_{j=1}^k\left(M_j N_j+t_j l_j L_j\right) .
	\end{equation}
	Now we define the sequences $\left\{n_j^{\prime}\right\},\left\{\beta_j^{\prime}\right\}$ and $\left\{\Delta_j^{\prime}\right\}$, by setting 
	$$
	\begin{aligned}
	& M_j^{\prime} : =M_k, \beta_j^{\prime}: =\beta_k, \Delta_j^{\prime}:=\Delta_k, \\
	&\textit{for } j=N_1+N_2+\cdots+N_{k-1}+t_1+\cdots+t_{k-1}+q \text { with } 1 \leq q \leq N_k, \\
\end{aligned}
$$
and
$$
	\begin{aligned}
	& M_j^{\prime} : =l_k L_k, \beta_j^{\prime} : =\beta_k, \Delta_j^{\prime} : =\left\{x_q^k\right\},\\
	&\textit{for } j=N_1+N_2+\cdots+N_k+t_1+\cdots+t_{k-1}+q \text { with } 1 \leq q \leq t_k.\\
	\end{aligned}
	$$
	Let
	$$
	\Theta_k : =\bigcap_{j=1}^k\left(\bigcup_{x_j \in \Delta_j^{\prime}} f^{-K_{j-1}} B_{M_j^{\prime}}\left(g; x_j, \beta_j^{\prime}\right)\right) \text { with } K_j : =\sum_{l=1}^j M_l^{\prime}, K_0 : =0.
	$$
     Then $\Theta_k$ is a non-empty closed set  by $g$-almost product property. We define a map
	$$\phi: \prod_{j\in\N}\Delta'_j \rightarrow \Theta,$$
    where $\Theta=\bigcap_{k \geq 1} \Theta_k.$ 
    %$\Theta$ is a closed set that is the disjoint union of non-empty closed sets $\Theta\left(x_1, x_2, \cdots\right)$ Labeled by $\left(x_1, x_2, \cdots\right)$ with $x_j \in \Delta_j^{\prime}$. Note that $\Theta$ is the intersection of closed sets. 
	We have the following claims:
	\begin{enumerate}
		\item $\phi$ is a bijection.
		
		\item $\Theta \subseteq G_K $.
		
		\item $\Theta\subseteq  \{x\in X:C_f(X)\subset \omega_f(x)\}$.
		
		\item 	$h_{top}^B(\Theta, f,\varepsilon^*) \geq H^*$.
	\end{enumerate}
	
	\textbf{Proof of Claim (1):}
 Assume that $x_j \neq y_j \in \Delta_j^{\prime}$ and $x,y\in \Theta$ such that
	%Assume $\Delta'_j=\Delta_k$.
	   $x \in f^{-K_{j-1}} B_{M_j^{\prime}}\left(g; x_j, \beta_j^{\prime}\right)$,  $y \in f^{-K_{j-1}} B_{M_j^{\prime}}\left(g; y_j, \beta_j^{\prime}\right)$.
	Since $x_j$ and $y_j$ are $\left(\frac{\delta^*}{4}, M_j^{\prime},\frac{\varepsilon^*}{4}\right)$-separated, by (\ref{eq-lowerbound-n2}) and (\ref{eq-beta}), there exists $ 0 \leq m \leq M'_k-1$ such that
	$d\left(f^m x_j, f^{m } y_j\right)>\frac{\varepsilon^*}{4}, d\left(f^m x_j, f^{m+K_{j-1} } x\right) \leq  \beta_j^{\prime}<\frac{\varepsilon^*}{16}, d\left(f^m y_j, f^{m+K_{j-1} } y\right) \leq \beta_j^{\prime}<\frac{\varepsilon^*}{16}.$
	Thus,
	$$d\left(f^{m+K_{j-1} } x, f^{m+K_{j-1} }  y\right) \geq d\left(f^m x_j, f^m y_j\right)-d\left(f^m x_j, f^{m+K_{j-1} }  x\right)-d\left(f^m y_j, f^{m+K_{j-1} }  y\right)>\frac{\varepsilon^*}{8}.$$
	then $x\neq y$.
	
	\qed

	\textbf{ Proof of Claim (2):} For any $k\in \mathbb{N}$.
	Define the stretched sequence $\left\{\alpha_n^{\prime}\right\}$ by
	$$
	\alpha_n^{\prime}: =\alpha_k \text { if } \sum_{j=1}^{k-1}\left(M_j N_j+t_j l_j L_j\right) < n \leq \sum_{j=1}^k\left(M_j N_j+t_j l_j L_j\right) .
	$$
	Then the sequence $\left\{\alpha_n^{\prime}\right\}$ has the same limit-point set as the sequence of $\left\{\alpha_k\right\}$. If
	$	\lim _{n \rightarrow \infty} d\left(\mathcal{E}_n(y), \alpha_n^{\prime}\right)=0$
	then the two sequences $\left\{\mathcal{E}_n(y)\right\}$ and $\left\{\alpha_n^{\prime}\right\}$ have the same limit-point set. 
	By (\ref{eq-N_k-1}),   $lim_{n\rightarrow \infty} \frac{K_{n+1}}{K_n} = 1.$	So from the definition of $\left\{\alpha_n^{\prime}\right\}$, we only need to prove
	$$\lim_{n\rightarrow+\infty}\rho(\mathcal{E}_{K_n}(y),\alpha_{K_n}')=0,$$
	for any $y \in \Theta$.
	Assume that $\sum_{j=1}^k\left(M_j N_j+t_j l_j L_j\right) < K_l \leq \sum_{j=1}^{k+1}\left(M_j N_j+t_j l_j L_j\right)$, hence $\alpha'_{K_l}=\alpha_{k+1}$.
	If $K_l \leq \sum_{j=1}^k\left(M_j N_j+\right.$ $\left.t_j l_j L_j\right)+M_{k+1} N_{k+1}$, by lemma \ref{lemma-estimate} and (\ref{eq-lowerbound-n1}),
	\begin{eqnarray*}
		d\left(\mathcal{E}_{K_l-\sum_{j=1}^k\left(M_j N_j+t_j l_j L_j\right)}\left(f^{\sum_{j=1}^k\left(M_j N_j+t_j l_j L_j\right)} y\right), \alpha_{k+1}\right) \leq 
		\xi_{k+1}+2 \beta_{k+1} .
	\end{eqnarray*}
	 If $K_l>\sum_{j=1}^k\left(M_j N_j+t_j l_j L_j\right)+M_{k+1} N_{k+1}$, by lemma \ref{lemma-estimate}, (\ref{eq-d}), (\ref{eq-lowerbound-n1}), (\ref{eq-lowerbound-n3}) and (\ref{eq-N_k-2}), we have
	\begin{equation}\label{eq-case2-1}
	\begin{aligned}
	& d\left(\mathcal{E}_{K_l-\sum_{j=1}^k\left(M_j N_j+t_j l_j L_j\right)}\left(f^{\sum_{j=1}^k\left(M_j N_j+t_j l_j L_j\right)} y\right), \alpha_{k+1}\right) \\
	& {\leq} \frac{M_{k+1} N_{k+1}}{K_l-\sum_{j=1}^k\left(M_j N_j+t_j l_j L_j\right)} d\left(\mathcal{E}_{M_{k+1} N_{k+1}}\left(f^{\sum_{j=1}^k\left(M_j N_j+t_j l_j L_j\right)} y\right), \alpha_{k+1}\right) \\
	& \ \ \ \ +\frac{K_l-\sum_{j=1}^k\left(M_j N_j+t_j l_j L_j\right)-M_{k+1} N_{k+1}}{K_l-\sum_{j=1}^k\left(M_j N_j+t_j l_j L_j\right)} \times 2 \\
	& {\leq} 1 \times\left(\xi_{k+1}+2 \beta_{k+1}\right)+\frac{2t_{k+1} l_{k+1} L_{k+1}}{M_{k+1} N_{k+1}} \\
	& {\leq} 3 \xi_{k+1}+2 \beta_{k+1}.
	\end{aligned}
	\end{equation}
	By Lemma \ref{lemma-estimate} and (\ref{eq-lowerbound-n1}),
	\begin{equation}\label{eq-case2-2}
	d\left(\mathcal{E}_{M_k N_k}\left(f^{\sum_{j=1}^{k-1}\left(M_j N_j+t_j l_j L_j\right)} y\right), \alpha_{k+1}\right) \leq \xi_k+2 \beta_k+d\left(\alpha_k, \alpha_{k+1}\right).
	\end{equation}
	Thus, by (\ref{eq-d}), (\ref{eq-case2-1}),  (\ref{eq-case2-2}),  (\ref{eq-N_k-2}),  (\ref{eq-N_k-1}),  and (\ref{eq-lowerbound-n3}), we obtain
	$$
	\begin{aligned}
	& d\left(\mathcal{E}_{K_l}(y), \alpha_{k+1}\right) \\
	& \leq \frac{\sum_{j=1}^{k-1}\left(M_j N_j+t_j l_j L_j\right)}{K_l} d\left(\mathcal{E}_{\sum_{j=1}^{k-1}\left(M_j N_j+t_j l_j L_j\right)}(y), \alpha_{K_l}^{\prime}\right) \\
	& +\frac{M_k N_k}{K_l} d\left(\mathcal{E}_{M_k N_k}\left(f^{\sum_{j=1}^{k-1}\left(M_j N_j+t_j l_j L_j\right)} y\right), \alpha_{k+1}\right)+\frac{t_k l_k L_k}{K_l} \rho(\mathcal{E}_{t_k l_k L_k}(f^{ \sum_{j=1}^{k-1}\left(M_j N_j+t_j l_j L_j\right)+M_kN_k}y), \alpha_{k+1}) \\
	& +\frac{K_l-\sum_{j=1}^k\left(M_j N_j+t_j l_j L_j\right)}{K_l}d\left(\mathcal{E}_{K_l-\sum_{j=1}^k\left(M_j N_j+t_j l_j L_j\right)}\left(f^{\sum_{j=1}^k\left(M_j N_j+t_j l_j L_j\right)} y\right), \alpha_{k+1}\right) \\
	&\overset{}{\leq}  \frac{\sum_{j=1}^{k-1}\left(M_j N_j+t_j l_j L_j\right)}{\sum_{j=1}^k\left(M_j N_j+t_j l_j L_j\right)} \times 2+1 \times\left(\xi_k+2 \beta_k+d\left(\alpha_k, \alpha_{k+1}\right)\right)+2 \times \frac{t_k l_k L_k}{K_l} 
	+3 \xi_{k+1}+2 \beta_{k+1} \\
	&\overset{}{\leq}  2\xi_k+\xi_k+2 \beta_k+d\left(\alpha_k, \alpha_{k+1}\right)+2\xi_k+2 \xi_{k+1}+2 \beta_{k+1} .
	\end{aligned}
	$$
	Since $\xi_k, \beta_k, d\left(\alpha_k, \alpha_{k+1}\right)$ all converge to zero as $k$ goes to zero, this proves item (2).
	
	\textbf{Proof of Claim (3):}
	Fix $x \in \Theta$. By construction, for any $k \geq 1$, there is $a=\sum_{j=1}^{k-1}(M_j N_j+t_jl_jL_j)+M_kN_k$ such that for any $j=1, \cdots, t_k$, there is $A^j \subseteq [0,l_k L_k-1]\cap \N$ such that
	$$
	\max \left\{d\left(f^{a+l+(j-1) l_k L_k} x, f^l x_j^k\right): l \in A^j\right\} \leq \beta_k.
	$$
	By (\ref{eq-AP2}), 
	$$
	\frac{\sharp A^j}{l_k L_k} \geq 1-\frac{g\left(l_k L_k\right)}{l_k L_k} \geq 1-\frac{1}{4 L_k} .
	$$
	Together with (\ref{eq-AP1}),   there is $p_j \in\left[0, l_k L_k-1\right]$ such that
	$$
	d\left(f^{a+p_j+(j-1) l_k L_k} x, f^{p_j} x_j^k\right) \leq \beta_k \text { and } d\left(x_j^k, f^{p_j} x_j^k\right) \leq \beta_k .
	$$
	This implies $d\left(f^{a+p_j+(j-1) l_k L_k} x, x_j^k\right) \leq 2 \beta_k$, so that the orbit of $x$ is $3 \beta_k$-dense in $C_f(X)$. Thus,
	$$\lim_{k\rightarrow +\infty} d\left(f^{a+p_j+(j-1) l_k L_k} x, x_j^k\right)=0,$$
	we obtain $\Theta \subseteq \{x\in X:C_f(X)\subset \omega_f(x)\}$ .

	\textbf{Proof of Claim (4):}
	From the proof of Claim (1), for any $k\in\mathbb{N}$, we know that $\Theta_k$ is a $\left(K_k, \frac{\varepsilon^*}{8}\right)$-separated set.
	%We will prove $ h_{\text {top }}^B(\Theta,f, \varepsilon) \geq h^*$.
	Define $$\mu_k=\frac{1}{\sharp \Delta_1^{\prime} \cdots \sharp \Delta_k^{\prime}} \sum_{x \in \Theta_k} \delta_x.$$
	Suppose $\mu=\lim _{n \rightarrow \infty} \mu_{k_n}$ for some $ k_n  \rightarrow \infty$.
	For any fix $l$ and all $p \geq 0$. Since $\mu_{l+p}\left(\Theta_{l+p}\right)=1$ and $\Theta_{l+p} \subset \Theta_l$,  we have $\mu_{l+p}\left(\Theta_l\right)=1$. Then $\mu\left(\Theta_l\right) \geq \limsup_{n \rightarrow \infty} \mu_{k_n}\left(\Theta_l\right)=1.$ Thereby,
	\begin{equation}\label{eq-mu1}
	\mu(\Theta)=\lim _{l \rightarrow \infty} \mu\left(\Theta_l\right)=1.
	\end{equation}
	For $k \geq 1, i=0,1,2, \cdots, N_k-1$, let
	$$
	n_{N_1+\cdots+N_{k-1}+i}=N_1+\cdots+N_{k-1}+t_1+\cdots+t_{k-1}+i
	$$
	for any $p \geq 1$, there is some $k$ so that $N_1+\cdots+N_{k-1}+t_1+\cdots+t_{k-1} \leq n_p \leq$ $N_1+\cdots+N_{k-1}+t_1+\cdots+t_{k-1}+N_k-1$.  Note that if $i <  N_k-1$, one has $n_{p+1}=n_p+1$, if $i = N_k-1$, one has $n_{p+1}=n_p+1+t_k$. 
	Then
	\begin{equation}\label{eq-K/K}
	\begin{split}
	1 \leq \frac{K_{n_{p+1}}}{K_{n_p}} &\leq \frac{K_{n_p}+\max \left\{M_k, M_k+t_k l_k L_k\right\}}{K_{n_p}}=1+\frac{M_k+t_k l_k L_k}{K_{n_p}}\\
	&\leq 1+\frac{M_k+t_k l_k L_k}{\sum_{j=1}^{k-1}\left(M_j N_j+t_j l_j L_j\right)} \overset{(\ref{eq-N_k-1})}{\leq} 1+\xi_{k-1},
	\end{split}
	\end{equation}
	and
	\begin{equation}\label{eq_sumnumber}
			\begin{aligned}
		& \prod_{j=n_p+1}^{n_{p+1}}\sharp \Delta_{j}^{\prime}=\sharp \Delta_k \overset{(\ref{eq-Delta_MK})}{\geq} \exp\{M_k H^*\} \overset{(\ref{eq-lowerbound-n4})}{\geq} \exp\{(H^*-\gamma)\left(M_k+t_k l_k L_k\right)\} \geq \exp\{(H^*-\gamma)\left(K_{n_{p+1}}-K_{n_p}\right)\} .
		\end{aligned}
	\end{equation}
   By the proof of claim (1),  let $ (x_1,\ldots,x_s,\ldots),\ (y_1,\ldots,y_s,\ldots)\in  \prod_{j\in\N} \Delta'_j$ such that $\phi( (x_1,\ldots,x_s,\ldots) )=x$  and $\phi( (y_1,\ldots,y_s,\ldots) )=y$. If
   $x_s \neq  y_s \in \Delta'_s$ for some $ s\in \mathbb{N}$,  then  $y \notin B_{K_s}(x, \varepsilon^*/8)$. 	For sufficiently large $m\in \N$, there exists a $n_p$ such that $K_{n_{p}} < m \leq K_{n_{p+1}}$. For  any $k_n\geq m$ and any $x\in \Theta$,  we have
	$$
	\begin{aligned}
	\mu_{k_n}\left(B_m(x,\frac{ \varepsilon^*}{8})\right) &\leq \mu_{k_n}\left(B_{K_{n_p}}(x,\frac{ \varepsilon^*}{8})\right) \\
	& \leq  \frac{\sharp \Delta_{n_{p}+1}'\cdots \sharp \Delta_{k_n}' }{\sharp \Delta_1^{\prime} \sharp \Delta_2^{\prime} \ldots \sharp \Delta_{k_n}^{\prime}} \\
	&= \frac{1}{\sharp\Delta_1^{\prime} \sharp \Delta_2^{\prime} \ldots \sharp \Delta_{n_{p}}^{\prime} }\\ &\overset{(\ref{eq_sumnumber})}{\leq} \exp\{-K_{n_{p}}(H^*-\gamma)\} \leq \exp\{-m (H^*-2\gamma)\}.
	\end{aligned}
	$$
	The last inequality is obtained from
  	$$\frac{K_{n_{p}}}{m}\geq \frac{K_{n_{p}}}{K_{n_{p+1}}}    \overset{(\ref{eq-K/K})}{\geq} \frac{1}{1+\xi_{k-1}} \overset{(\ref{eq-xi})}{\geq}  \frac{H^*-2\gamma}{H^*-\gamma} .$$
	Then
	\begin{equation}\label{eq-muB}
\mu\left(B_m(x, \frac{ \varepsilon^*}{8})\right) \leq \liminf_{n \rightarrow \infty} \mu_{k_n}\left(B_m(x, \frac{ \varepsilon^*}{8})\right) \leq \exp\{-m  ( \inf_{\mu \in K} \inf _{\operatorname{diam}(\xi)<4\varepsilon^*} h_\mu(f, \xi) -3 \gamma )\} .
	\end{equation}
	
	The following version of the Entropy Distribution Principle is needed to derive the entropy formula.
	\begin{Lem}\cite[Lemma 13]{BR2023}\label{lemma_entropy_distribution_principle}
		Let $f: X \rightarrow X$ be a continuous transformation and $\varepsilon>0$. Given a set $Z \subset X$ and a constant $s \geq 0$, suppose there exist a constant $C>0$ and a Borel probability measure $\mu$ satisfying $\mu(Z)>0$ and $\mu\left(B_n(x, \varepsilon)\right) \leq C \exp\{-n s\}$ for every ball $B_n(x, \varepsilon)$ such that $B_n(x, \varepsilon) \cap Z \neq \emptyset$.
		Then $h^B_{top}(Z, f, \varepsilon) \geq s$.
	\end{Lem}
	Combine  Lemma \ref{lemma_entropy_distribution_principle} and (\ref{eq-muB}), we  get
	$$h^B_{top}(\Theta, f,\frac{ \varepsilon^*}{8}) \geq \inf_{\mu \in K} \inf _{\operatorname{diam}(\xi)<4\varepsilon^*} h_\mu(f, \xi) -3\gamma .$$
	By Claim (2) and Claim (3), it holds that
	\begin{equation}\label{eq-lowebound-h}
	h_{top}^B(G^C_K, f,\frac{\varepsilon^*}{8}) \geq  h_{top}^B(\Theta, f,\frac{\varepsilon^*}{8})\geq \inf_{\mu \in K} \inf _{\operatorname{diam}(\xi)<4\varepsilon^*} h_\mu(f, \xi) - 3\gamma .
	\end{equation}
	Thus, 
	$$
	\begin{aligned}
		S-3\gamma &\overset{(\ref{eq-lowebound-1})}{\leq} \frac{\inf_{\mu \in K} \inf _{\operatorname{diam}(\xi) <4\varepsilon^*} h_\mu(f, \xi) - 3\gamma  }{|\ln 4 \varepsilon^*|} \overset{(\ref{eq-lowebound-h})}{\leq} \frac{h_{top}^B(G^C_K, f,\frac{\varepsilon^*}{8})}{|\ln\frac{\varepsilon^*}{8}|}	\frac{|\ln \frac{\varepsilon^*}{8}|}{|\ln 4 \varepsilon^*|}\\ &\overset{(\ref{eq-lowerbound-2})}{\leq} 	
		\overline{\operatorname{mdim}}^B_{M}\left(G^C_K, f, d\right)+\gamma.\\
	\end{aligned}
	$$
	Since $\gamma$ is arbitrary, we get the conclusion
	$$	\overline{\operatorname{mdim}}^B_{M}\left(G^C_K, f, d\right)\geq \limsup _{\varepsilon \rightarrow 0} \frac{1}{|\ln \varepsilon|} \inf_{\mu \in K} \inf _{\operatorname{diam}(\xi)<\varepsilon} h_\mu(f, \xi).$$
	
	$\underline{\operatorname{mdim}}^B_{M}\left(G^C_K, f, d\right)\geq \liminf_{\varepsilon \rightarrow 0} \frac{1}{|\ln \varepsilon|} \inf_{\mu \in K} \inf  _{\operatorname{diam}(\xi)<\varepsilon} h_\mu(f, \xi)$ can be obtained by same method, we omit the proof here.

  \end{proof}

\subsection{Proof of Corollary \ref{cor-maintheorem}}
	By Theorem \ref{MainTheorem},   it follows that
	$$
	\overline{\operatorname{mdim}}^B_{M}\left(G_\mu, f, d\right)
	\geq \overline{\operatorname{mdim}}^B_{M}\left(G^C_\mu, f, d\right)=\limsup _{\varepsilon \rightarrow 0} \frac{1}{|\ln \varepsilon|} \inf _{\operatorname{diam}(\xi)<\varepsilon} h_\mu(f, \xi).
	$$

	Combine with Lemma \ref{lemma-leq2}(3), we get the conclusion that $$\overline{\operatorname{mdim}}^B_{M}\left(G_\mu, f, d\right) =\limsup _{\varepsilon \rightarrow 0} \frac{1}{|\ln \varepsilon|} \inf _{\operatorname{diam}(\xi)<\varepsilon} h_\mu(f, \xi).$$
	Bowen lower metric mean dimension formula can be obtained by similar method.
\qed

\section{Applications: Multifractal analysis}\label{Section_5}
In this section,  we give an abstract framework on multifractal analysis for possibly applicability to more general functions.
Let $\alpha: \mathcal{M}_f(X) \rightarrow \mathbb{R}$ be a continuous function,
here we list three conditions for $\alpha$.
\begin{description}
	\item[C.1] For any $\mu, \nu \in \mathcal{M}_f(X), \beta(\theta)=\alpha(\theta \mu+(1-\theta) \nu)$ is strictly monotonic on $[0,1]$ when $\alpha(\mu) \neq \alpha(\nu)$.
	\item[C.2] For any $\mu, \nu \in  \mathcal{M}_f(X), \beta(\theta)=\alpha(\theta \mu+(1-\theta) \nu)$ is constant on [0,1] when $\alpha(\mu)=\alpha(\nu)$.
	\item[C.3] For any $\mu, \nu \in  \mathcal{M}_f(X), \beta(\theta)=\alpha(\theta \mu+(1-\theta) \nu)$ is not constant over any subinterval of $[0,1]$ when $\alpha(\mu) \neq \alpha(\nu)$ ( Note that C.1 implies C.3).
\end{description}

The function $\alpha$ can be defined as:
\begin{enumerate}
	\item  $\alpha\equiv 0$ (Satisfying condition C.2).
	\item  Let $\phi, \psi$ be two continuous functions on X and $\psi$ required to be positive. Define $\alpha(\mu)=\frac{ \int \phi d\mu}{\int \psi d \mu}$ (Satisfying condition C.1 and C.2 \cite[Lemma 3.2]{DHT}). Specially,  $\psi=1$. 
	\item  $\alpha(\mu)=\lim_{n\rightarrow \infty} \frac{1}{n} \int \varphi_n d\mu$ with asymptotically additive sequences of continuous functions $\Phi=(\varphi_n)_{n\in \N}$ 	(Satisfying condition C.1 and C.2). Then $\alpha$ is a continuous function \cite{FH2010}. Furthermore, it is affine. 
\end{enumerate}

Let 
$$L_\alpha : =\left[\inf _{\mu \in \mathcal{M}_f(X)} \alpha(\mu), \sup _{\mu \in \mathcal{M}_f(X)} \alpha(\mu)\right]$$
and $\operatorname{Int}\left(L_\alpha\right)$ denote its interior interval. For any $a\in L_{\alpha}$ and any $\varepsilon>0$, define
$$
\begin{gathered}
I_\alpha : =\left\{x \in X: \inf _{\mu \in V_f(x)} \alpha(\mu)<\sup _{\mu \in V_f(x)} \alpha(\mu)\right\}, R_\alpha(a) : =\left\{x \in X: \inf _{\mu \in V_f(x)} \alpha(\mu)=\sup _{\mu \in V_f(x)} \alpha(\mu)=a\right\}, \\
R_{\alpha}^C(a) : = R_{\alpha}(a)\cap  \left\{x\in X: C_f(X) \subset \omega_f(x)  \right\},  \ I^C_\alpha : =I_\alpha \cap \left\{x\in X:  C_f(X) \subset  \omega_f(x) \right\},\\
 R_\alpha : =\left\{x \in X: \inf _{\mu \in V_f(x)} \alpha(\mu)=\sup _{\mu \in V_f(x)} \alpha(\mu)\right\}=\bigcup_{a\in L_{\alpha}} R_{\alpha}(a) , \\
\mathcal{M}_f(X,\alpha,a) : = \{ \mu \in \mathcal{M}_f( X) : \alpha(\mu)=a \},\\
H_\alpha(a,\varepsilon) : =  \frac{1}{|\ln \varepsilon|} \sup_{\mu \in \mathcal{M}_f(X,\alpha,a)} \inf _{\operatorname{diam}(\xi)<\varepsilon} h_\mu(f, \xi),
\end{gathered}
$$
and recall that $\mathcal{M}_f(X,\varphi, a) = \{ \mu \in \mathcal{M}_f( X) : \int \varphi d\mu=a\}$.

\subsection{The variational principle of level sets}

\begin{Thm}\label{Thorem_level}
Suppose $(X, f)$ is a dynamical system and $\alpha:  \mathcal{M}_f(X) \rightarrow \mathbb{R}$ is a continuous function.
	
	\begin{enumerate}
			\item[(1)] For any $\mu  \in  \mathcal{M}_f(X)$, one has $\overline{\operatorname{mdim}}_M^B(G_\mu^C,f,d)= \limsup_{\varepsilon \rightarrow 0}\frac{1}{|\ln \varepsilon|}\inf_{\operatorname{diam} (\xi)<\varepsilon}h_\mu(f,\xi)$. Then for any real number $a \in L_\alpha$, the set $R_\alpha^C(a)$ is not empty and
		$$
		\overline{\operatorname{mdim}}^B_M\left(R_\alpha(a),f,d\right)=	\overline{\operatorname{mdim}}^B_M\left(R_\alpha^C(a) ,f,d\right)=\limsup_{\varepsilon\rightarrow 0}H_\alpha(a,\varepsilon).
		$$
		If further $f$ has positive metric mean dimension and $\operatorname{Int}\left(L_\alpha\right) \neq \emptyset$, then for any real number $a \in \operatorname{Int}\left(L_\alpha\right)$, one has
		$$
		\overline{\operatorname{mdim}}^B_M\left(R_\alpha(a),f,d\right)=	\overline{\operatorname{mdim}}^B_M\left(R_\alpha^C(a),f,d\right)=\limsup_{\varepsilon\rightarrow 0}H_\alpha(a,\varepsilon)>0 .
		$$
		\item[(2)] If $a \in L_\alpha \backslash \operatorname{Int}\left(L_\alpha\right)$, then the set $R_\alpha(a)$ is not empty and
		$$
		\overline{\operatorname{mdim}}^B_M\left(R_\alpha(a),f,d\right)=\limsup_{\varepsilon\rightarrow 0}H_\alpha(a,\varepsilon).
		$$
		
			\item[(3)] $\overline{\operatorname{mdim}}^B_M\left(X,f,d\right) = \overline{\operatorname{mdim}}^B_M\left(R_\alpha,f,d\right)$. If further $\operatorname{Int}\left(L_\alpha\right) \neq \emptyset$ and $\alpha$ satisfies C.3, then
		$$
		\overline{\operatorname{mdim}}^B_M\left(X,f,d\right) =\sup _{a \in \operatorname{Int}\left(L_\alpha\right)} \limsup_{\varepsilon\rightarrow 0}H_\alpha(a,\varepsilon)=\overline{\operatorname{mdim}}^B_M\left(R_\alpha(a),f,d\right).
		$$
	\end{enumerate}
\end{Thm}

\begin{remark}
	A similar result can be obtained for the lower metric mean dimension by replacing $\overline{\operatorname{mdim}}_M^B$ with $\underline{\operatorname{mdim}}_M^B$. For brevity, we omit the detailed statement.
\end{remark}

\begin{remark}
	Corollary \ref{cor_multifractual} (1) and (2) are  special cases of Theorem \ref{Thorem_level} when $\alpha(\mu)=\int \varphi d\mu$.
\end{remark}

\begin{proof}
		(1)
		On the one hand, for any invariant measure $\mu$ with $\alpha(\mu)=a$, note that $G_\mu^C \subseteq R_\alpha^C(a) $. Then by the assumption, for any $\gamma$, there exists a sufficient small $\varepsilon>0$ such that $$\overline{\operatorname{mdim}}_M^B(G_\mu^C,f,d) \geq  \frac{1}{|\ln \varepsilon|}\inf _{\operatorname{diam}(\xi)<\varepsilon} h_\mu(f, \xi)-\gamma.$$
		Thus, we have
		$$\overline{\operatorname{mdim}}^B_{M}\left( R^C_\alpha(a),f,d\right) \geq \overline{\operatorname{mdim}}^B_{M}\left( G^C_\mu,f,d\right)\geq  \frac{1}{|\ln \varepsilon|}\sup_{\mu \in \mathcal{M}_f(X,\alpha,a)}\inf_{\operatorname{diam}(\xi)<\varepsilon} h_\mu(f, \xi)-\gamma.$$
		Then
		$$ \overline{\operatorname{mdim}}^B_{M}\left( R_\alpha(a) ,f,d\right)  \geq \overline{\operatorname{mdim}}^B_{M}\left( R^C_\alpha(a) ,f,d\right) \geq \limsup_{\varepsilon \rightarrow 0}H_{\alpha}(a,\varepsilon ).$$
	On the other hand,	let
	$$
	^{\mathcal{M}_f(X,\alpha,a)}G=\left\{x \in X:\left\{\mathcal{E}_n(x)\right\} \text { has all its limit points in } \mathcal{M}_f(X,\alpha,a)\right\} .
	$$
   Note that $R_{\alpha}(a)=^{\mathcal{M}_f(X,\alpha,a)}G \subset G^{\mathcal{M}_f(X,\alpha,a)} $, by Lemma \ref{lemma-leq2} (1), for any $\varepsilon>0$,  we can see that
	$$
	 h^B_{\text {top}}\left(R_{\alpha}^C(a),f,\varepsilon \right)  \leq h^B_{\text {top}}\left(R_{\alpha}(a),f,\varepsilon \right) \leq \sup_{\mu \in \mathcal{M}_f(X,\alpha,a)} \inf _{\operatorname{diam}(\xi)<\varepsilon} h_\mu(f, \xi).
	$$
 Then
	$$\overline{\operatorname{mdim}}^B_{M}\left( R^C_\alpha(a),f,d\right) \leq \overline{\operatorname{mdim}}^B_{M}\left( R_\alpha(a),f,d\right)   \leq  \limsup_{\varepsilon\rightarrow 0}H_\alpha(a,\varepsilon).$$

	If further $f$ has positive metric mean dimension and $\operatorname{Int}\left(L_{\alpha}\right) \neq \emptyset$, fix $a \in \operatorname{Int}\left(L_{\alpha}\right)$, by  Theorem \ref{Thm_variationprinciple}, there exist a sequence $\{\varepsilon_j\}_{j\in \N} \rightarrow 0$ and an invariant measure $\mu_1$, such that 
	$$ \frac{1}{|\ln\varepsilon_j|}\inf_{\operatorname{diam}(\xi)<\varepsilon_j}h_{\mu_1}(f,\xi)>0.$$ 
	If $\alpha\left(\mu_1\right)=a$, then $\limsup_{\varepsilon \rightarrow 0}H_\alpha(a,\varepsilon)>0$. If $\alpha\left(\mu_1\right) \neq a$, without loss of generality, we may assume that $\alpha\left(\mu_1\right)<a$. Since $a \in \operatorname{Int}\left(L_{\alpha}\right)$, we can take another invariant measure $\mu_2$ such that $\alpha\left(\mu_2\right)>a$. Then one can take a suitable $\theta \in(0,1)$ such that $\mu=\theta \mu_1+(1-\theta) \mu_2$ satisfies that $\alpha(\mu)=a$. By the affine property of $h_{\mu}(f,\xi)$, we have $$\frac{1}{|\ln\varepsilon_j|}\inf_{\operatorname{diam}(\xi)<\varepsilon_j}h_\mu(f,\xi) \geq   \frac{\theta}{|\ln\varepsilon_j|}\inf_{\operatorname{diam}(\xi)<\varepsilon_j} h_{\mu_1}(f,\xi)>0,$$ then $\limsup_{\varepsilon\rightarrow 0}H_\alpha(a,\varepsilon)>0$.

	(2) $\overline{\operatorname{mdim}}^B_M\left(R_\alpha(a),f,d\right)\leq \limsup_{\varepsilon\rightarrow 0}H_\alpha(a,\varepsilon)$ is the same as the proof of item (1). To prove the opposite inequality, fix an invariant measure $\mu$ with $\alpha(\mu)=a$. By the ergodic decomposition theorem, for any $\varepsilon>0$, there exists an ergodic measure $\nu$ (as one ergodic component) such that $\alpha(\nu)=a$ and $h_\nu(f,\xi)>h_\mu(f,\xi)-\varepsilon$ for any partition $\xi$ of X. Note that $\nu\left(G_\nu\right)=1$, by Lemma \ref{lemma_G_mu},  there is a $\varepsilon'>0$ such that for any $0<\varepsilon<\varepsilon'$,  one has
	$$h^B_{top} \left(G_\nu,f,\varepsilon\right) \geq  \inf_{\operatorname{diam}(\xi)<\varepsilon}h_\nu(f,\xi)-\varepsilon.$$
    Note that $G_\nu \subseteq R_\alpha(a)$, thus 
    $$
    \begin{aligned}
    h^B_{top} \left(R_{\alpha}(a),f,\varepsilon\right)
      \geq  h^B_{top} \left(G_\nu,f,\varepsilon\right
   )  \geq  \inf_{\operatorname{diam}(\xi)<\varepsilon}h_\nu(f,\xi)-\varepsilon > \inf_{\operatorname{diam}(\xi)<\varepsilon}h_\mu(f,\xi)-2\varepsilon.
    \end{aligned}
    $$ 
    Then $ h^B_{top} \left(R_{\alpha}(a),f,\varepsilon\right) \geq \sup_{\mu\in  \mathcal{ M}_f(X,\alpha,a)}\inf_{\operatorname{diam}(\xi)<\varepsilon}h_\mu(f,\xi)-2\varepsilon.$
    Divided by $|\ln \varepsilon|$, we obtain
    $$\overline{\operatorname{mdim}}^B_M\left(R_\alpha(a),f,d\right)\geq \limsup_{\varepsilon\rightarrow 0}H_\alpha(a,\varepsilon).$$

	(3) Note that $\cup_{\mu\in \mathcal{ M}_f(X)}G_{\mu} \subseteq R_\alpha$, so $\mu\left(R_\alpha\right)=1$ for any invariant measure $\mu$. By  Lemma \ref{lemma_G_mu},  there is a $\varepsilon'>0$ such that for any $0<\varepsilon<\varepsilon'$, 
	$$ \inf_{ \operatorname{diam} (\xi)<\varepsilon}h_{\mu}(f,\xi)\leq h_{top}^B(R_{\alpha}, f, \varepsilon)+\varepsilon.$$
	Thus,
	$$\overline{\operatorname{mdim}}^B_M \left(R_\alpha,f,d\right) \geq \limsup_{\varepsilon\rightarrow 0} \frac{1}{|\ln \varepsilon|}\sup_{\mu\in\mathcal{M}_f(X)} \inf_{\operatorname{diam}(\xi)<\varepsilon}h_\mu(f,\xi).$$
	 By Theorem \ref{Thm_variationprinciple}, we get
	$$\overline{\operatorname{mdim}}^B_M \left(R_\alpha,f,d\right) \geq \limsup_{\varepsilon\rightarrow 0} \frac{1}{|\ln \varepsilon|}\sup_{\mu\in\mathcal{ M}_f(X)} \inf_{\operatorname{diam}(\xi)<\varepsilon}h_\mu(f,\xi)= \overline{\operatorname{mdim}}^B_M \left(X,f,d\right).$$
	Thus, $$\overline{\operatorname{mdim}}^B_M \left(R_\alpha,f,d\right)  = \overline{\operatorname{mdim}}^B_M \left(X,f,d\right).$$
	What's more, for any $a\in \operatorname{Int}(L_{\alpha})\neq \emptyset$, one has $\limsup_{\varepsilon\rightarrow 0}H_\alpha(a,\varepsilon) \leq \overline{\operatorname{mdim}}^B_M \left(X,f,d\right)$, Thus, $$\overline{\operatorname{mdim}}^B_M \left(R_\alpha,f,d\right) \geq \sup_{a\in \operatorname{Int}(L_{\alpha})}\limsup_{\varepsilon\rightarrow 0}H_\alpha(a,\varepsilon).$$	
	 Now we only need to prove $\overline{\operatorname{mdim}}^B_M \left(R_\alpha,f,d\right) \leq \sup_{a\in \operatorname{Int}(L_{\alpha})}\limsup_{\varepsilon\rightarrow 0}H_\alpha(a,\varepsilon).$ 
	By  Theorem \ref{Thm_variationprinciple}, there exist an $\varepsilon>0$ and  an invariant measure $\mu$, such that 
	$$\frac{1}{|\ln \varepsilon
		|}\inf_{ \operatorname{diam} (\xi)<\varepsilon} h_\mu(f,\xi)>\overline{\operatorname{mdim}}^B_M \left(X,f,d\right) -\varepsilon.$$
	If $\alpha(\mu) \in \operatorname{Int}\left(L_\alpha\right)$, take $\omega=\mu$. Otherwise, take an invariant measure $\nu$ such that $\alpha(\nu)\neq \alpha(\mu)$ and $\alpha(\nu) \in \operatorname{Int}(L_{\alpha}).$
     By condition C.3, one can choose $\theta \in(0,1)$ close to 1 such that $\omega=\theta \mu+(1-\theta) \nu$ satisfies $\alpha(\omega) \in \operatorname{Int}\left(L_\alpha\right)$ and
      $$
        \frac{1}{|\ln \varepsilon
       	|}\inf_{ \operatorname{diam}( \xi)<\varepsilon} h_\omega(f,\xi) \geq 
        \theta \frac{1}{|\ln \varepsilon
      	|}\inf_{ \operatorname{diam} (\xi)<\varepsilon} h_\mu(f,\xi)>\overline{\operatorname{mdim}}^B_M \left(X,f,d\right) -\varepsilon.
     $$
      Thus, $\sup _{a \in \operatorname{Int}\left(L_\alpha\right)} \limsup_{\varepsilon\rightarrow 0}H_\alpha(a,\varepsilon) \geq \limsup_{\varepsilon\rightarrow 0} H_\alpha(\alpha(\omega),\varepsilon)>\overline{\operatorname{mdim}}^B_M \left(X,f,d\right) -\varepsilon$.
\end{proof}

\subsection{Bowen metric mean dimension of irregular sets}

\begin{Thm}\label{Thorem_irregular}
Suppose $(X, f)$ is a dynamical system and $\alpha:  \mathcal{M}_f(X) \rightarrow \mathbb{R}$ is a continuous function.
	Assume that for any  $\{\mu_1,\cdots \mu_m\}\subseteq \mathcal{M}_f(X)$ and any compact connected non-empty subset $K\subset \operatorname{conv}\{\mu_1,\cdots ,\mu_m\}$, one has
	$$
	\overline{\operatorname{mdim}}^B_{M}\left(G^C_K, f, d\right)=\limsup _{\varepsilon \rightarrow 0} \frac{1}{|\ln \varepsilon|} \inf_{\mu \in K} \inf _{\operatorname{diam}(\xi)<\varepsilon} h_\mu(f, \xi),
	$$
	  and $\inf _{\mu \in \mathcal{M}_f(X)} \alpha(\mu)< \sup _{\mu \in \mathcal{M}_f(X)} \alpha(\mu)$. Then
	  $I_\alpha^C\neq \emptyset$, Moreover,
	
	    	\item[(1)] If $f$ has positive  metric mean dimension, then $$\overline{\operatorname{mdim}}^B_M\left(I_\alpha,f,d\right)\geq 	\overline{\operatorname{mdim}}^B_M\left(I^C_\alpha ,f,d\right)>0.$$

		\item[(2)] If $\alpha: \mathcal{M}_f( X) \rightarrow \mathbb{R}$ satisfies C.3, then
		$$	\overline{\operatorname{mdim}}^B_M\left(I_\alpha,f,d\right)=	\overline{\operatorname{mdim}}^B_M\left(I^C_\alpha,f,d\right)=	\overline{\operatorname{mdim}}^B_M\left(X,f,d\right). $$

\end{Thm}
\begin{remark}
	A similar result can be obtained for the lower metric mean dimension by replacing $\overline{\operatorname{mdim}}_M^B$ with $\underline{\operatorname{mdim}}_M^B$. For brevity, we omit the detailed statement.
\end{remark}

\begin{remark}
	Let $\alpha(\mu)=\int \varphi d \mu$, we get item (3) in Corollary \ref{cor_multifractual} .
\end{remark}

\begin{proof}
(1) Take $\mu_1, \mu_2$ with $\alpha\left(\mu_1\right)<\alpha\left(\mu_2\right)$ and let $K=\left\{t \mu_1+(1-t) \mu_2 : t \in[0,1]\right\}.$ By assumption, $G_K^C \neq \emptyset$. Note that $G_K^C \subseteq I_\alpha^C $, thus, $I_\alpha^C$ is not empty. 
 If $\overline{\operatorname{mdim}}^B_M\left(X,f,d\right)>0$,
fix $\varepsilon \in(0, \overline{\operatorname{mdim}}^B_M(X,f,d ))$, by Theorem \ref{Thm_variationprinciple}, we can take an invariant measure $\nu$ such that 
\begin{equation}\label{eq_I_1}
	\frac{1}{|\ln \varepsilon
		|}\inf_{ \operatorname{diam}( \xi)<\varepsilon} h_\nu(f,\xi)>\overline{\operatorname{mdim}}^B_M \left(X,f,d\right) -\varepsilon>0.
\end{equation}
    By assumption we can take another invariant measure $\nu_1$ such that $\alpha\left(\nu_1\right) \neq \alpha(\nu)$. Then by continuity of $\alpha$ there is $\theta \in(0,1)$ such that $\rho =\theta \nu+(1-\theta) \nu_1$ satisfies $\alpha(\rho) \neq \alpha(\nu)$. Then
    $$\inf_{ \operatorname{diam} (\xi)<\varepsilon} h_\rho(f,\xi) \geq \theta \inf_{ \operatorname{diam} (\xi)<\varepsilon} h_\nu(f,\xi)>0.$$ 
    Let $K=\{t \nu+(1-t) \rho : t \in[0,1]\}$.  
    Then by assumption, one can get
    \begin{equation}\label{eq_I_2}
    	\overline{\operatorname{mdim}}^B_M \left(G_K^C,f,d\right)=\limsup _{\varepsilon \rightarrow 0} \frac{1}{|\ln \varepsilon|} \min\left\{\inf_{ \operatorname{diam} (\xi)<\varepsilon}h_\nu(f,\xi), \inf_{ \operatorname{diam} (\xi)<\varepsilon}h_\rho(f,\xi)\right\}>0.
    \end{equation}
    Note that 
    $G_K^C \subseteq I_\alpha^C$. Thus, $$\overline{\operatorname{mdim}}^B_M \left(I_{\alpha}^C,f,d\right) \geq \overline{\operatorname{mdim}}^B_M \left(G_K^C,f,d\right)>0.$$

	(2) Obviously, $ \overline{\operatorname{mdim}}^B_M (I_{\alpha}^C,f,d) \leq \overline{\operatorname{mdim}}^B_M (I_{\alpha},f,d) \leq \overline{\operatorname{mdim}}^B_M \left(X,f,d\right),$ so we only need to prove the reverse inequality.
   If  $\overline{\operatorname{mdim}}^B_M\left(X,f,d\right)=0$, then the result  is trivial. We may assume $\overline{\operatorname{mdim}}^B_M\left(X,f,d\right)>0$. If further $\alpha: \mathcal{M}_f(X) \rightarrow \mathbb{R}$ satisfies C.3, then above $\theta \in(0,1)$ can be chosen very close to 1 such that $\rho =\theta \nu+(1-\theta) \nu_1$ satisfies that 
	$$\frac{1}{|\ln \varepsilon|}\inf_{ \operatorname{diam} (\xi)<\varepsilon} h_\rho(f,\xi) \geq \theta\frac{1}{|\ln \varepsilon|} \inf_{ \operatorname{diam} (\xi)<\varepsilon} h_\nu(f,\xi) >\overline{\operatorname{mdim}}^B_M \left(X,f,d\right) -\varepsilon.$$
	 Then combine with (\ref{eq_I_1}) and (\ref{eq_I_2}), we obtain
	 $$
	 \begin{aligned}
	 \overline{\operatorname{mdim}}^B_M \left(I_{\alpha}^C,f,d\right) &\geq \overline{\operatorname{mdim}}^B_M \left(G_K^C,f,d\right)=\limsup _{\varepsilon \rightarrow 0} \frac{1}{|\ln \varepsilon|} \min\left\{\inf_{ \operatorname{diam} (\xi)<\varepsilon}h_\nu(f,\xi), \inf_{ \operatorname{diam} (\xi)<\varepsilon}h_\rho(f,\xi)\right\}\\
	 &>\overline{\operatorname{mdim}}^B_M \left(X,f,d\right) -\varepsilon.
	 \end{aligned}
	$$ 
	Then by the arbitrary of $\varepsilon$, one has  $ \overline{\operatorname{mdim}}^B_M (I_{\alpha}^C,f,d) \geq  \overline{\operatorname{mdim}}^B_M \left(X,f,d\right).$
	
\end{proof}

 \section*{ Acknowledgments.}
I would like to acknowledge the editors and reviewers for their diligent work and valuable suggestions. My special thanks go to XiaoBo Hou and Xue Liu for their extensive and insightful advice on this paper.

\end{document}